\documentclass[12pt]{amsart}

\usepackage[utf8]{inputenc}

\usepackage{hyperref}
\usepackage{graphicx}
\usepackage{amssymb}
\usepackage{float}
\usepackage{fullpage}
\usepackage{physics}
\usepackage{comment}
\usepackage{tikz}
\usetikzlibrary{positioning, arrows.meta, shapes.geometric}
\usepackage{mathtools}
\usepackage{bm}
\newtheorem{theorem}{Theorem}[section]
\newtheorem{lemma}[theorem]{Lemma}

\newtheorem{proposition}[theorem]{Proposition}

\theoremstyle{definition}

\theoremstyle{remark}
\newtheorem{remark}{Remark}

\numberwithin{equation}{section}
\numberwithin{remark}{section}

\newcommand{\cond}{\operatorname{cond}}
\newcommand{\B}{\mathcal{B}}

\newcommand{\C}{{\mathbb C}}
\newcommand{\R}{{\mathbb R}}

\newcommand{\Q}{{\mathbb Q}}
\newcommand{\Z}{{\mathbb Z}}

\newcommand{\eps}{\varepsilon}

\newcommand{\CC}{{\mathfrak C}}

\newcommand{\SL}{\operatorname{SL}}

\newcommand{\GL}{\operatorname{GL}}

\renewcommand{\a}{\mathfrak{a}}

\newcommand{\g}{\mathrm{g}}

\renewcommand{\d}{\mathrm d}

\newcommand{\sgn}{\operatorname{sgn}}
\renewcommand{\pmod}[1]{\,\,(\mathrm{mod}\,{#1})}

\begin{document}
\title{On primes represented by $aX^2+bY^3$}
\author{Jori Merikoski}
\address{Mathematical Institute,
University of Oxford,
Andrew Wiles Building,
Radcliffe Observatory Quarter,
Woodstock Road,
Oxford,
OX2 6GG}
\email{jori.merikoski@maths.ox.ac.uk}
\subjclass[2020]{}
\begin{abstract}
Let $a,b>0$ be coprime integers. Assuming a conjecture on Hecke eigenvalues along binary cubic forms, we prove an asymptotic formula for the number of primes of the form $ax^2+by^3$ with $x \leq X^{1/2}$ and $y  \leq X^{1/3}$. The proof combines sieve methods with the theory of real quadratic fields/indefinite binary quadratic forms, the Weil bound for exponential sums, and spectral methods of GL(2) automorphic forms. We also discuss applications to elliptic curves.
\end{abstract}

\maketitle
\tableofcontents

\section{Introduction}

The problem of counting prime numbers along thin polynomial sequences has been solved only for a very narrow class of polynomials. Results fall into one of two lineages, one starting from the Friedlander-Iwaniec  primes of the form $X^2+Y^4$ \cite{FI} and one from the Heath-Brown primes of the form $X^3+2Y^3$ \cite{hb}. See \cite{fouvryi,FIalmost,greensawhney,hbli,lds,merikoskisparse,merikoskipolyprimes,pratt,stanleyFI} and \cite{hbmoroz1,hbmoroz2,li,maynard} for their respective descendants.  In particular, all known results about prime values of thin polynomials require that the polynomial factorizes in a number field gaining at least one linear variable, for example, $X^2+Y^4= (X+iY^2)(X-iY^2)$ in $\Q(i)$ and $X^3+2Y^3$ is the norm of $X+Y\sqrt[3]{2}$ in $\Q(\sqrt[3]{2})$.

We introduce a new approach for counting the prime values of the polynomial $aX^2+bY^3$ that has no obvious factorization. In the absence of a linear variable, we will leverage the quadratic variable. This is a much harder task and our main result is conditional on a hypothesis that we formulate now.

Let  $d >0$ be a fundamental discriminant and consider the real quadratic field $\Q(\sqrt{d})$ (see Section \ref{sec:binaryquadr} for a more detailed discussion). We let $\lambda_{\chi \xi^\ell}(n) = \sum_{N_{\Q(\sqrt{d})} (\a) = n} \chi \xi^\ell (\a) $ denote the Hecke eigenvalue associated to the Gr\"o\ss encharakter $\chi \xi^\ell$,  parametrized by  $\ell \in \Z$ and class group characters $\chi \in \widehat{\CC(d)}$ \cite{hecke1,hecke2}. For $\chi \xi^{\ell}=1$ it is given by the Dirichlet convolution $\lambda_{1} = 1\ast (\tfrac{d}{\cdot})$ where  $(\tfrac{d}{n})$ is the Jacobi symbol. For $T> 1$ we define the truncated approximation $\lambda_{1}^\sharp(n,T) = 1\ast( (\tfrac{d}{\cdot})\mathbf{1}_{[0,T]})(n)$  and denote the error term by $\lambda_{1}^\flat(n,T) =\lambda_{1}(n)-  \lambda_{1}^\sharp(n,T)$. Having no reason to suspect otherwise, for $\chi \xi^{\ell} \neq 1$ we expect square-root cancellation along the values of binary cubic forms, and for $a \in \Z_{>0}$ and $\eps > 0$ we make the following hypothesis. \\

\noindent\textbf{Conjecture $\mathrm{C}_a(\eps)$}.
Let  $C(X,Y) =c_1X^3 - c_2Y^3 \in \Z[X,Y]$  for some $c_1 ,c_2 > 0$. Then for $\chi \xi^{\ell} \neq 1$ and $B_1 ,B_2 >1$ we have
\begin{align*}
\sum_{\substack{y_1 \leq B_1, \, y_2 \leq B_2 \\ C(y_1,y_2) \equiv 0 \pmod{a}}} \lambda_{\chi \xi^{\ell} }(|\tfrac{1}{a}C(y_1,y_2)|) &\ll_\eps \max\{B_1^2,B_2^2,c_1^2,c_2^2,d,|\ell|^2\}^\eps (B_1+B_2), 
\end{align*}
and for some $\eta > 0$ we have for any $T \in [(B_1+B_2)^{1-\eta},B_1+B_2]$
\begin{align*}
  \sum_{\substack{y_1 \leq B_1, \, y_2 \leq B_2 \\ C(y_1,y_2) \equiv 0 \pmod{a}}} \lambda^\flat_{1}(|\tfrac{1}{a}C(y_1,y_2)|,T) & \ll \max\{c_1^2,c_2^2,d,|\ell|^2\}^\eta (B_1^2+B_2^2)^{1-\eta}.
\end{align*}

For $B_1 \approx B_2$ taking $\eps =o(1)$ corresponds to square-root cancellation, and $\eps=1/2$ corresponds to no cancellation at all. For $\chi \xi^{\ell}=1$ we require only a small power-saving. This conjecture serves as a placeholder, we only require such a bound on average over a large family of eigenvalues, see Conjecture $\mathrm{L}_{a,b}(\eps)$ in Remark \ref{remark:largesieve}. Assuming that this conjecture holds for all $\eps > 0$, we can show an asymptotic formula for the number of primes of the form $aX^2+bY^3$. As usual, $\Lambda(n)$ denotes the von Mangoldt function.
\begin{theorem} \label{thm:asymp}
    Let $a,b > 0$ be coprime integers. Assume that Conjecture $\mathrm{C}_a(\eps)$ holds for all $\eps >0$.  Then 
    \begin{align*}
        \sum_{x \leq X^{1/2}} \sum_{y \leq X^{1/3}} \Lambda(ax^2+by^3) = (1+o(1)) X^{5/6}.
    \end{align*}
\end{theorem}
For a fixed $\eps > 0$ the error term $o(1)$ is replaced by $O(\eps)$ (see Theorem \ref{thm:technicalasymp}). We get a correct order lower bound for primes with the fixed value $\eps=1/17$ and a lower bound for products of exactly $k \geq 2$ primes for any $\eps < 1/4$. 
\begin{theorem} \label{thm:lower}
    Let $a,b > 0$ be coprime integers. Assuming that Conjecture $\mathrm{C}_a(1/17)$ holds, we have
    \begin{align*}
        \sum_{x \leq X^{1/2}} \sum_{y \leq X^{1/3}} \Lambda(ax^2+by^3) \geq (0.05+o(1)) X^{5/6}.
    \end{align*}
    Assuming that Conjecture $\mathrm{C}_a(\eps)$ holds for some $\eps <1/4$, we have for any $k \geq 2$
      \begin{align*}
        \sum_{x \leq X^{1/2}} \sum_{y \leq X^{1/3}} (\underbrace{\Lambda \ast \cdots \ast \Lambda}_{k \, \mathrm{times}}) (ax^2+by^3) \gg X^{5/6} (\log X)^{k-1}.
    \end{align*}
\end{theorem}

The proof of Theorem \ref{thm:asymp} may be adapted  to show that, under the same hypothesis, the Möbius function $\mu(n)$ has cancellation.
\begin{theorem} \label{thm:mobius}
       Let $a,b > 0$ be coprime integers. Assume that Conjecture $\mathrm{C}_a(\eps)$ holds for all $\eps >0$.  Then 
        \begin{align*}
        \sum_{x \leq X^{1/2}} \sum_{y \leq X^{1/3}} \mu(ax^2+by^3) = o(X^{5/6}).
    \end{align*}
\end{theorem}
Unconditionally,  for the Liouville function $\lambda$, Teräväinen \cite[Theorem 2.12]{ter} has shown that $ \lambda(ax^2+by^3)$ takes both values $\pm 1$ infinitely often. 

\subsection{Applications to elliptic curves}
A major motivation for studying prime numbers of this form comes from elliptic curves
\begin{align*}
    E_{A,B}: \quad y^2 = x^3+Ax+B
\end{align*}
whose discriminant 
\begin{align*}
    \Delta(E_{A,B}) = -16(4 A^3 + 27 B^2) 
\end{align*}
controls the places of bad reduction. Therefore, as a special case of Theorem \ref{thm:asymp} with $a=27,b=4$, assuming Conjecture $\mathrm{C}_{27}(\eps)$ for all $\eps > 0$, we get an asymptotic formula for the number of elliptic curves with exactly one place of bad reduction $p > 2$.

It is not clear if the argument can be adapted to study the distribution of the root numbers of the elliptic curves $E_{A,B}$ (see \cite[proof of Proposition 3.1] {youngelliptic} for $4A^3+27B^2$ square-free) 
\begin{align*}
  w(E_{A,B}) = - \mu(4A^3+27B^2)(\tfrac{6 B}{4 A^3+27 B^2}) w_2(E_{A,B}).
\end{align*}
The twist by the Jacobi symbol appears surprisingly difficult to accommodate in our approach. Helfgott \cite[Corollary 5.2]{helfgott} has shown that for the family $y^2=x(x+A)(x+B)$ the root numbers are evenly distributed, by considering essentially $(\tfrac{B}{A}) \mu(A B(A-B))$, but for his arguments the Jacobi symbol causes only minor issues.

It should however be possible to consider the root numbers of the quadratic twist family
\begin{align*}
    E^{(B)}_{A,B} : \quad By^2 = x^3 + Ax + B
\end{align*}
for $B$ and $4 A^3+27 B^2$ square-free with $\gcd(B,\Delta(E_{A,B}))=1$, since their root numbers satisfy 
\begin{align*}
     w(E^{(B)}_{A,B}) = w(E_{A,B})(\tfrac{-B}{4 A^3+27 B^2}) =  - \mu(4A^3+27B^2)(\tfrac{-6}{4 A^3+27 B^2}) w_2(E_{A,B}).
\end{align*}
The local root number $w_2(E_{A,B})$ may be controlled using \cite[Table III]{rizzo} after sorting $A$ and $B$ into residue classes.

\subsection{Overview}
We restrict to the case $a=b=1$  in this non-rigorous sketch. Replacing the rough cut-offs by smooth weights $f$, we consider counting primes weighted by the sequence
\begin{align*}
    a_n = \sum_{n=x^2+y^3} f(\tfrac{x}{A})  f(\tfrac{y}{B}), \quad A\asymp X^{1/2}, \, B \asymp X^{1/3}.
\end{align*}
By a sieve argument (essentially Heath-Brown's identity \cite{hbidentity}), the proof reduces to the asymptotic evaluation of three different types of sums
    \begin{align*}
        &\text{Linear sums (Type I)}  &&\sum_{d \leq D_1} \alpha_d \sum_{n} a_{dn},\\
           &\text{Bilinear sums (Type II)}   && \hspace{-8pt}\sum_{m \sim X/N}\alpha_m  \sum_{n} \beta_n a_{mn},\\
            &\text{Divisor sums (Type I$_2$)} &&\sum_{d \leq D_2} \alpha_d \sum_{m,n} a_{d m n},
\end{align*}
where $\alpha$ and $\beta$ denote bounded complex coefficients. 
\subsubsection{Type I sums}
We are able to handle the Type I sums for
\begin{align*}
 D_1< X^{5/9}   
\end{align*}
by applying Poisson summation to the variables $x$ and $y$ and the Weil bound for exponential sums.
\subsubsection{Type II sums}
Assuming conjecture $\mathrm{C}_1(\eps)$ for some $\eps>0$, we are able to handle Type II sums in the range 
\begin{align*}
    X^{1/6}  < N  < X^{1/3-2\eps/3}.
\end{align*}
This estimate for type II sums is based on the following considerations that can be seen as the main idea of this paper. After applying the Cauchy-Schwarz inequality on $m$ and rearranging sums, the task reduces to evaluating a sum of the form
\begin{align*}
    \sum_{n_1,n_2 \sim N} \beta_{n_1} \beta_{n_2} \sum_{k \asymp NX} \Upsilon_{n_1,n_2}(k) \mathcal{Q}_{n_1,n_2}(k),
\end{align*}
where $\Upsilon_{n_1,n_2}(k),\mathcal{Q}_{n_1,n_2}(k)$ denote the  restricted representations by binary cubic and quadratic forms
\begin{align*}
    \Upsilon_{n_1,n_2}(k) = \sum_{k=n_1 y_2^3-n_2y_1^3} f(\tfrac{y_1}{B}) f(\tfrac{y_2}{B}) \quad \text{and} \quad 
       \mathcal{Q}_{n_1,n_2}(k) = \sum_{k=n_2 x_1^2-n_1x_2^2} f(\tfrac{x_1}{A}) f(\tfrac{x_2}{A}).
\end{align*}
Using the classical correspondence between the class group $\CC(d)$ of $\Q(\sqrt{d})$ and the set of $\GL_2(\Z)$-equivalence classes of binary quadratic forms of discriminant $d,$ the latter may be expanded (see Lemma \ref{le:binaryrep}) using the Hecke characters for the real quadratic field of discriminant $d=4 n_1 n_2 \asymp N^2$. We essentially get
\begin{align*}
    \mathcal{Q}_{n_1,n_2}(k) \approx \frac{1}{h(d) R_d}\sum_{|\ell| \ll R_d} \sum_{\chi \in \widehat{\CC(d)}} \chi \xi^\ell(\mathfrak{n}_2) \lambda_{\chi \xi^{\ell} } (k), \quad \mathfrak{n}_2^2 = (4 n_2) \subseteq O_d.
\end{align*}
Here the class number $h(d)$ and the regulator $R_d$ satisfy the Dirichlet class number formula \eqref{eq:classnumberformula}
\begin{align*}
    2h(d)R_d =\sqrt{d} L(1,(\tfrac{d}{\cdot})) \approx \sqrt{d} \asymp N,
\end{align*}
where  $\approx$ holds on average over $n_1,n_2$. By the Cohen-Lenstra heuristics, we expect that typically $h(d) \ll 1$ and $R_d \asymp N$, so that the characters $\xi^\ell$ are the main culprit for the losses. Corresponding to the fact that a minimal solution to Pell's equations may be exponentially large, we expect that often $\mathcal{Q}_{n_1,n_2}(k) = 0.$ More precisely, we have an expansion of the function $\mathcal{Q}_{n_1,n_2}(k)$ which has expected density $\approx 1/N$ into roughly $N$ many harmonics $\lambda_{\chi \xi^{\ell} } (k)$. 

Invoking Conjecture $\mathrm{C}_1(\eps)$, we obtain for $\chi \xi^{\ell} \neq 1$
\begin{align*}
    \sum_{k \asymp NX}  \Upsilon_{n_1,n_2}(k)  \lambda_{\chi \xi^{\ell} } (k) \ll X^{1/3+2\eps/3},
\end{align*}
which is sufficient provided that $N < X^{1/3-2\eps/3}$. We merely require this bound on average over the complete family of $n_1,n_2 \sim N$, $\chi \in \widehat{\CC(d)}$, and $|\ell| \ll R_d$ (see Remark \ref{remark:largesieve}). It is tempting to think that this problem would be amenable to large sieve techniques but unfortunately, the cubic forms appearing in $\Upsilon_{n_1,n_2}(k)$ are entangled with $n_1,n_2$. We mention that similar sums for a fixed discriminant and a fixed binary cubic or quartic form arise naturally in the context of Manin's conjecture for Ch\^atelet surfaces \cite{Woo}.

The task is then to evaluate the main term coming from $\chi \xi^{\ell}=1$
\begin{align*}
    \sum_{n_1,n_2 \sim N} \frac{\beta_{n_1} \beta_{n_2}}{ \sqrt{n_1n_2} L(1,(\tfrac{4n_1n_2}{\cdot}))}  \sum_{k \asymp NX} \Upsilon_{n_1,n_2}(k)  \lambda_{1}(k).
\end{align*}
It is not obvious how to make the argument unconditional and we have to assume that $\lambda_{1}(k)$ may be replaced by the truncated approximation $\lambda^{\sharp}_{1}(k,T)$.   For the divisor function along binary cubic forms Greaves has shown an asymptotic formula with a power saving \cite{greaves}. 

Using Heath-Brown's large sieve for quadratic characters \cite{hbquad}, on average over $n_1,n_2$ we can replace the factor $ L(1,(\tfrac{4n_1n_2}{\cdot}))^{-1}$ by $\sum_{k \leq X^{\eps}} \frac{\mu(k)}{k} (\tfrac{4n_1n_2}{k})$  (see Lemma \ref{le:Linverse}). For the 
truncated approximation $ \lambda^\sharp_{1}(k,T)=\sum_{c \leq T} (\tfrac{4n_1n_2}{c})$ we can evaluate the sum over $y_1,y_2$ by applying Poisson summation twice, which produces a count for cubic congruences
\begin{align} \label{eq:cubicintro}
    \# \{ (y_1,y_2) \in (\Z/c\Z)^2: n_1 y_2^3 \equiv n_2 y_1^3 \pmod{c} \}.
\end{align}
Expanding \eqref{eq:cubicintro} by cubic Dirichlet characters, the principal characters give the main term. We then need to bound error terms of the form 
\begin{align*}
    \frac{1}{N}\sum_{c \leq T} \frac{1}{c}\sum_{k \leq X^{\eps}} \frac{\mu(k)}{k} \sum_{\substack{\chi \pmod{c} \\ \chi^3 = \chi_0 \neq \chi_0 }} \bigg| \sum_{n \sim N} \beta_n \chi(n) (\tfrac{n}{ck})\bigg|^2.
\end{align*}
This is bounded using the large sieve for sextic characters due to Baier and Young \cite{baieryoungcubic}.

\subsubsection{Type I$_2$ sums}
The obtained Type I and Type II ranges would already be enough for a lower bound of the correct order of magnitude for the number of primes $p=x^2+y^3$. To show an asymptotic formula, we also need to consider the Type I$_2$ sums. We are able to handle these for 
\begin{align*}
  D_2 < X^{1/4}  
\end{align*}
by using the spectral methods of GL(2) automorphic forms. The proof will appear in an upcoming joint work with Grimmelt \cite{GMquadratic}, as an application of the \emph{averages over orbits} introduced in \cite{GM}. To sketch the idea in the critical case $d \sim D_2$, $m \sim n \sim \sqrt{X/D_2}$, note that for a fixed $y$ we consider a variant of the divisor problem along a quadratic polynomial $\sum_{x} f(\tfrac{x}{A}) d(x^2+y^3)$. Applying Poisson summation on $x$ produces the expected main term and an error term with the dual variable of length $H \approx  \sqrt{D_2}$. 

Arguing similarly to Duke, Friedlander, and Iwaniec \cite{DFIprimes}, utilizing the factorization $y^3=y \cdot y^2$,  we get a spectral expansion for the error term that is morally of the shape (suppressing the Eisenstein series and other minor details)
\begin{align*}
 \frac{B^{1/2}}{H^{1/2}}  \sum_{y} f(\tfrac{y}{B}) \sum_{d \sim D_2} \alpha_d \sum_{\substack{j \\ |t_j| \ll 1 }}^{\Gamma_0(d)}   \lambda_j(y)  \sum_{h \sim H} \frac{\lambda_j(h)}{d^{1/2}}\sum_{z \in \Lambda_y} \sum_{\tau } u_j(\tau z).
\end{align*}
Here $u_j(z)$ are the ($L^2$-normalized) Maass cusp forms for the  Hecke congruence subgroup $\Gamma_0(d)$ and $\lambda_j(h)$ denote the Hecke eigenvalues. Here $\Lambda_y \subseteq \mathbb{H}$ is essentially the set Heegner points for the discriminant $y$ so that $\# \Lambda_y \approx y^{1/2}$, and $\tau$ runs over a subset of size $\ll d^{o(1)}$ in $\Gamma_0(d) \backslash \SL_2(\Z)$. 

Applying the Cauchy-Schwarz inequality with the variables $z,\tau$ on the outside (as in \cite{DFIprimes}), and using the Rankin-Selberg bound $\sum_{y} f (\tfrac{y}{B})|\lambda_j(y)|^2 \leq (dB)^{o(1)}B$, we would get the range $D_2 < X^{1/6}$. This turns out to be just barely insufficient for an asymptotic formula for primes. Taking advantage of the average over the orbits $z \in \Lambda_y$ allows us to essentially save  a factor of $(\#\Lambda_y)^{1/2} \approx X^{1/12}$, which bumps the range up to $D_2 < X^{1/4}$. We note that any range with  $D_2 > X^{1/6+\eta}$ would be sufficient for getting the asymptotic formula -- curiously, the prime detecting sieve has a discontinuity at $1/6$ in terms of the parameter $\frac{\log D_2}{\log X}$, caused by crossing from Type I$_2$ sums to Type II sums (see \cite{fordmaynard} for a detailed discussion of discontinuities in sieve methods). 

It seems difficult to improve the error term $o(1)$ in Theorem \ref{thm:asymp}. Even if we assumed a much more uniform conjecture with $\max\{\cdots\}^\eps$ replaced by $(\log \max\{\cdots\})^{O(1)}$, we would only improve the error term $o(1)$ to $O(\frac{\log \log X}{\log X})$. We are missing arithmetic information for multiple different types of sums, especially (i) three variables of size $X^{1/3+o(1)}$  and (ii) six variables of size $X^{1/6+o(1)}$.  

\subsubsection{Generalizations to other sequences}
We mainly leverage the fact that the first term is a quadratic monomial, so the discussion, at least in principle, extends primes of the form $X^2+y$ with $y$ weighted by other sequences $\gamma_y$. Of particular interest are sequences $\gamma_y$ with support of size $X^{1/2-\eta}$ for some small $\eta>0$. For such a sequence the diagonal terms in the Type II argument are admissible if $N > X^{\eta}$. Therefore, to produce a non-trivial Type II range, we would only require a small amount of saving instead of square-root saving in the corresponding convolution sums along Hecke eigenvalues
\begin{align} \label{eq:gammmasumintro}
    \sum_{v_1,v_2} \gamma_{v_1} \gamma_{v_2} \lambda_{\chi \xi^{\ell} }(n_1 v_2 - n_2 v_1).
\end{align}
Furthermore, since $n_1,n_2$ are small, the uniformity in the conductor does not seem to be a formidable issue for, say, automorphic techniques. It would therefore be of great interest to obtain bounds for sums of the form \eqref{eq:gammmasumintro}, as it would quickly translate to results about primes or at least products of $k\geq 2$ primes of the form $x^2+y$ with the weights $\gamma_y$.

\subsection{Acknowledgements}
I am grateful to James Maynard, Kyle Pratt, and Lasse Grimmelt for helpful discussions. The project has
received funding from the European Research Council (ERC) under the European Union's Horizon 2020 research and innovation programme (grant agreement No 851318).
\section{Restricted representations by indefinite binary quadratic forms} \label{sec:binaryquadr}
\subsection{Real quadratic fields}
Let $d> 0$  be a fundamental discriminant, that is, $d \equiv 0,1 \pmod{4}$ and $d$ is square-free except for a possible factor of $4$ or $8$. Consider the real quadratic field $\Q(\sqrt{d}) := \Q(\sqrt{d})$ and its ring of integers 
\begin{align*}
    O_d = \Z[\sqrt{\omega_d}], \quad  \text{where} \quad \omega_d := \begin{cases}
        \tfrac{1}{2}\sqrt{d}, \quad & d \equiv 0 \pmod{4} \\
        \tfrac{1}{2}(1+\sqrt{d}), \quad &d \equiv 1 \pmod{4}.
    \end{cases} 
\end{align*}
For $z = x+y\sqrt{d}, x,y \in \Q$ we let $z^\sigma := x-y\sqrt{d}$ denote the other embedding into $\R$. The norm $N_d:\Q(\sqrt{d}) \to \Q$ is defined by $ N_d (z):= z z^\sigma = x^2-dy^2$. The norm is extended to ideals $\a \subseteq O_d$ via $N_d(\a) = \# O_d/\a$, so that for any principal ideal $\a=(z)$ we have $N_d(\a) = |N_d(z)|$.

The group of units has rank 1 and it is generated by $-1$ and the fundamental unit $\eps_d$, which is defined as the smallest element $\eps \in O_d$ with $\eps > 1$ and $N_d(\eps) = \pm 1$. In other words, $\eps_d = \frac{a+b\sqrt{d}}{2}$ is the smallest element $> 1$ such that $(a,b) \in \Z$ is a solution to the Pell equations $a^2-db^2 = \pm 4$.

We let $\mathcal{I}_d$ denote the group of fractional ideals and $\mathcal{P}_d$ denote the subgroup of principal fractional ideals.  The ideal class group is then defined as
\begin{align*}
    \CC(d) := \mathcal{I}_d/\mathcal{P}_d. 
\end{align*} 
The class group $\CC(d)$ is a finite abelian group and the class number is defined as $h(d) := \# \CC(d).$ The regulator of the number field $\Q(\sqrt{d})$ is defined by  $R_d := \log \eps_d.$ We then have the \emph{Dirichlet class number formula} for real quadratic fields
\begin{align} \label{eq:classnumberformula}
   2 h(d)  R_d = \sqrt{d}  \, L(1,(\tfrac{d}{\cdot})), \quad L(s,(\tfrac{d}{\cdot})) := \sum_{n=1}^\infty(\tfrac{d}{n}) n^{-s}
\end{align}
where $(\tfrac{d}{n})$ is the primitive real character associated to the fundamental discriminant $d$. 

We define the hyperbolic coordinates of any non-zero $z \in \Q(\sqrt{d})$ via
\begin{align*}
r(z) = \sgn(z)|N_d(z)|^{1/2}, \quad   \alpha(z) = \tfrac{1}{2}\log |z/z^\sigma|.
\end{align*}
Then analogous to the polar coordinates for imaginary quadratic fields we have for $z z^\sigma > 0$
\begin{align*}
    z& =x+y\sqrt{d}=  r(z)e^{\alpha(z)}, \quad z^\sigma = r(z)e^{-\alpha(z)},\\
    x&= r(z) \cosh \alpha(z), \quad y = \frac{1}{\sqrt{d}}  r(z) \sinh \alpha(z).
\end{align*}
For $z z^\sigma < 0$ we have $ z^\sigma = -r(z)e^{-\alpha(z)}$ and the roles of $\cosh \alpha$ and $\sinh \alpha$ are swapped.

The Hecke characters \cite{hecke1,hecke2} $\{\xi^\ell\}_{\ell\in \Z}$ are defined for principal ideals $\a=(z)$ by the formula
\begin{align*}
  \xi^\ell(\a) = \xi^\ell(z) :=  e(\ell \tfrac{\alpha(z)}{R_d}), \quad e(x)= e^{2\pi i x}.
\end{align*}
It is quick to check using $\eps_d^\sigma = 1/\eps_d$ that this definition does not depend on the choice of the generator $z$. The character $\xi^\ell$ may be extended to a character on all fractional ideals, and the extension is unique up to multiplication by a class group character $\chi \in  \widehat{\CC(d)}.$

We denote
\begin{align}
    \psi_{\chi \xi^{\ell} ,s}(\a) :=  \xi^\ell(\a) \chi(\a) N_{d}(\a)^{-s/2}, \quad   \psi_{\chi \xi^{\ell} } :=   \psi_{\chi \xi^{\ell} ,0}.
\end{align}
We define the associated  Hecke eigenvalues
   \begin{align} \label{eq:hecekdef}
         \lambda_{\chi \xi^{\ell} }(n) := \sum_{\substack{\a \subseteq O_d \\ N_d(\a) = n}} \psi_{\chi \xi^{\ell} }(\a).
    \end{align}
For $\chi \xi^{\ell}=1$ we have the Dirichlet convolution $ \lambda_{1} =1\ast(\tfrac{d}{\cdot})$. For any $T>1$, we define the truncated approximation 
\begin{align*}
      \lambda^\sharp_{1}(n)= \lambda^\sharp_{1}(n,T) = \sum_{\substack{c|n \\ c\leq T}} (\tfrac{d}{c})
\end{align*}
and set
\begin{align} \label{eq:heckeflatdef}
   \begin{split}
       \lambda^\flat_{1}(n)&=\lambda^\flat_{1}(n,T) :=   \lambda_{1}(n) -  \lambda^\sharp_{1}(n,T), \quad 
         \lambda^\flat_{\chi \xi^{\ell} }(n) :=    \lambda_{\chi \xi^{\ell} }(n)   \quad \text{for} \quad \chi \xi^{\ell} \neq 1.
   \end{split} 
\end{align}
For $n=0$ we set $\lambda^\sharp_{1}(0) =\lambda^\flat_{1}(0) = \lambda_{\chi \xi^{\ell} }(0) = 0$.

The associated $L$-function
\begin{align*}
    L(s,\chi\xi^\ell) = \sum_n \frac{\lambda_{\chi \xi^{\ell} }(n)}{n^s}
\end{align*}
has a meromorphic continuation \cite{hecke1,hecke2}, with a simple pole only for $(\chi \xi^{\ell} ,s) = (1,1)$. It satisfies a functional equation of the form 
\begin{align*}
    L(s,\chi\xi^\ell) =  (d/\pi^2)^{1/2-s} G(1-s,\chi \xi^{\ell} ) L(1-s,\overline{\chi\xi^\ell}) , 
\end{align*}
where for $\sigma < \eta < 0$ \cite{colemannormforms,dukemultid,rademacher}
\begin{align*}
     |G(1-s,\chi \xi^{\ell} )|  \ll_\eta (1+t^2+\ell^2)^{1/2-\sigma}.
\end{align*}
In particular, for $\sigma < \eta < 0$ by the functional equation this implies
\begin{align} \label{eq:Lfuncbound}
    L(\sigma+it,\chi\xi^\ell)  \ll_\eta (d (1+t^2+\ell^2))^{1/2-\sigma}.
\end{align}

\subsection{Indefinite binary quadratic forms}
Let $d>0$ be a fundamental disrciminant and let $\B_d$ denote the set of binary quadratic forms $B \in \Z[X,Y]$ 
\begin{align*}
    &B(X,Y) = aX^2 +bXY +cY^2 \quad \text{of discriminant} \quad b^2-4ac=d.
\end{align*}
 The group $\GL_2(\Z)$ acts on $\B_d$ via
\begin{align*}
    B^\gamma (X,Y) = B(a_{11} X+a_{12} Y,a_{21} X+a_{22} Y), \quad \text{where} \quad \gamma =\mqty(a_{11} &a_{12} \\a_{21} & a_{22}) \in \GL_2(\Z).
\end{align*}
There is a well-known isomorphism between the $\GL_2(\Z)$ equivalence classes of such forms and the class group $\CC(d)$, which sends the class  of a form $B(X,Y) = aX^2 +bXY +cY^2 \in \B_d$ to the class in $\CC(d)$ of the ideal 
\begin{align*}
    (a,\tfrac{b+\sqrt{d}}{2})_\Z = \{ 2a x + \tfrac{b+\sqrt{d}}{2} y: x,y\in \Z \} \subseteq O_d.
\end{align*}

We now restrict to the specific binary quadratic forms $B(X,Y) = n_2 X^2-n_1 Y^2$ of discriminant $4n_1n_2$, for which we can argue directly. We obtain an expansion for representations restricted by a smooth weight on the variables.
\begin{lemma} \label{le:binaryrep}
Let $\epsilon \in \{\pm \}$.    Let $n_1,n_2 > 0$ be square-free, odd, and coprime and denote $d=4n_1n_2$. Let  $\mathfrak{n}_{2} \subseteq O_{d}$ denote the ideal such that $(4n_2)=\mathfrak{n}_2^2$. Let $\delta,\delta_0,\eta,Z>0$ and let $F:(0,\infty)^2 \to \C$ be a fixed smooth function  $F$ is supported on $(\eta,1/\eta)^2$ and $\epsilon(x_1^2-x_2^2) > \delta_0$. Suppose that for all $J \geq 0$ we have $\partial_{x_1}^{J_1}\partial_{x_2}^{J_2} F(x_1,x_2) \ll_{J_1,J_2} \delta^{-J_1-J_2}$. Then for any  $m \in \Z_{>0}$ we have
    \begin{align*}
        \sum_{\substack{x_1,x_2 \in \Z \\ |n_2 x_1^2-n_1 x_2^2| = m}} F(\tfrac{x_1 2 n_2}{Z},\tfrac{x_2 \sqrt{4 n_1n_2}}{Z}) 
       &= \frac{1}{h(d)R_d}\frac{1}{2\pi i}\sum_{\ell \in \Z} \sum_{\chi \in \widehat{\CC(d)}}   \int_{(0)} Z^{s}\check{F}(\ell,s)    \psi_{\chi \xi^{\ell} ,s}( \mathfrak{n}_2)\lambda_{\chi \xi^{\ell} ,s}(m) \d s, 
    \end{align*}
    where the coefficients $\check{F}(\ell,s)  $  satisfy the decay property that for any $J \geq 0$ 
    \begin{align} \label{eq:Fdecay}
        |\check{F}(\ell,s)| \ll_{\eta,J} \frac{1+ |\log \delta_0|}{(1+  \delta |\ell| /R_d +\delta |s|)^{J} }.
    \end{align}
\end{lemma}
\begin{proof}
We consider $\epsilon=+$, the case $\epsilon=-$ is similar. Since $4n_2(n_2 x_1^2-n_1 x_2^2)= (2n_2x_1)^2-4n_1n_2 x_2^2$, the representations $|n_2 x_1^2-n_1 x_2^2| = m$ are in 1-to-1 correspondence with  $z=x_1+ x_2 \sqrt{d} \in O_d$ which satisfy $|N_d(z)|=4 n_2 m$ and $2 n_2  | x_1$.  The condition $2n_2|x_1$ may be dropped since it is implied by $
  4n_2|  |N_d(z)| = |x_1^2-4n_1n_2x_2^2| $. By unique prime factorization of ideals we have $\mathfrak{n}_2 | (z)$. Let $f(r,\alpha)$ be defined by $f(r,\alpha)=F(x_1,x_2 \sqrt{d})$ for $z=r e^\alpha \in \Q(\sqrt{d})$. We then have 
    \begin{align*}
        \sum_{\substack{x_1,x_2 \in \Z \\ |n_2 x_1^2-n_1 x_2^2| = m}} F(\tfrac{x_1 2 n_2}{ Z},\tfrac{x_2 \sqrt{4n_1n_2}}{Z})  = &\sum_{\substack{z \in O_d  \\ |N_d(z)| = 4n_2 m } } F(\tfrac{x_1}{ Z},\tfrac{x_2 \sqrt{d}}{Z})
        = \sum_{\substack{\a=(z) \\ N_d(\a) = 4n_2 m }}  \sum_{k \in \Z} f(\tfrac{r(z)}{Z},\alpha(z) + k R_d).
        \\
    \end{align*}
   Denoting $G(r, \alpha) = \sum_{k \in \Z} f(r,\alpha + k R_d)$, applying Mellin inversion and Fourier series expansion to $G(r/Z, \alpha)$, and expanding the condition that $\a$ is a principal ideal by class group characters  we get
\begin{align*}
        \sum_{\substack{x_1,x_2 \in \Z \\ |n_2 x_1^2-n_1 x_2^2| = m}} F(\tfrac{x_1 2 n_2}{ Z},\tfrac{x_2 \sqrt{4n_1n_2}}{Z})   =& \frac{1}{h(d) R_d}\frac{1}{2\pi i}\sum_{\ell \in \Z} \sum_{\chi \in \widehat{C}(d)}   \int_{(0)} Z^{s}\check{F}(\ell,s)    \sum_{\substack{\a  \in O_d \\ N_d(\a) =4 n_2 m }} \psi_{\chi \xi^{\ell} ,s}(\a) \d s \\
       =& \frac{1}{h(d) R_d}\frac{1}{2\pi i}\sum_{\ell \in \Z} \sum_{\chi \in \widehat{C}(d)}   \int_{(0)} Z^{s}\check{F}(\ell,s)    \psi_{\chi \xi^{\ell} ,s}(\mathfrak{n}_2)\sum_{\substack{\a  \in O_d \\ N_d(\a) = m  }} \psi_{\chi \xi^{\ell} ,s}(\a) \d s.
\end{align*}
By unfolding the sum over $k \in \Z$ we have
  \begin{align*}
    \check{F}(\ell,s)  &=  \int_0^{R_d}\int_0^\infty  G(r,\alpha) e(-\ell \tfrac{\alpha}{R_d})r^{s} \frac{\d r}{r} \d \alpha =\int_{\R}\int_0^\infty  F(r \cosh \alpha, r \sinh \alpha ) e(-\ell \tfrac{\alpha}{R_d})r^{s} \frac{\d r}{r} \d \alpha.
\end{align*}
The decay property \eqref{eq:Fdecay} follows from iterating integration by parts with respect to the symmetric differential operator
\begin{align*}
    \Xi_{r,\alpha} = 1+ (\tfrac{\delta}{  2\pi i} \partial_\alpha)^2 + (i\delta  |r| \partial_r)^2,
\end{align*}
that is, for any $J \geq 0$ we have
\begin{align*}
      \check{F}(\ell,s)  &=  \frac{1}{(1+(\delta \ell /R_d)^2 + (\delta |s|)^2)^J} \int_{\R}\int_0^\infty   F(r \cosh \alpha, r \sinh \alpha ) \Xi_{r,\alpha}^J\bigg( e(-\ell \tfrac{\alpha}{R_d})r^{s} \bigg) \frac{\d r}{r} \d \alpha \\
       &= \frac{1}{(1+(\delta \ell /R_d)^2 + (\delta |s|)^2)^J} \int_{\R}\int_0^\infty   \Xi_{r,\alpha}^J \bigg(F(r \cosh \alpha, r \sinh \alpha )\bigg)  e(-\ell \tfrac{\alpha}{R_d})r^{s} \frac{\d r}{r} \d \alpha  \\
      & \ll_{\eta,J} \frac{1+ |\log \delta_0|}{(1+  \delta |\ell| /R_d +\delta |s|)^{J} }. 
\end{align*}
\end{proof}
We also need the following variant, where the smooth weight depends only on the hyperbolic angle. The proof is similar but easier.
\begin{lemma} \label{le:binaryrepangle}
    Let $n_1,n_2 > 0$ be square-free, odd, and coprime and denote $d=4n_1n_2$. Let  $\mathfrak{n}_{2} \subseteq O_{d}$ denote the ideal such that $(4n_2)=\mathfrak{n}_2^2$. Let $\delta,K>0$ and let $f:\R \to \C$ be a smooth function supported in $[-K,K]$ satisfying $\partial_x^J f(x) \ll_J \delta^{-J}$. Then for any  $m \in \Z_{>0}$ we have
    \begin{align*}
        \sum_{\substack{x_1,x_2 \in \Z  \\ |n_2 x_1^2-n_1 x_2^2| = m}} f(\log| \tfrac{x_1 2 n_2 + x_2\sqrt{4 n_1 n_2}}{x_1 2 n_2-x_2\sqrt{4 n_1 n_2} }|) 
       &= \frac{1}{h(d)R_d}\sum_{\ell \in \Z} \sum_{\chi \in \widehat{\CC(d)}} \check{f}(\ell)   \psi_{\chi \xi^{\ell} }( \mathfrak{n}_2)\lambda_{\chi \xi^{\ell} }(m),
    \end{align*}
    where the coefficients $\check{f}(\ell)  $  satisfy  $|\check{f}(\ell)| \ll_{J} K (1+  \delta |\ell| /R_d )^{-J}$ for all $J>0$.
\end{lemma}
\section{Lemmas}
We need the following truncated version of the Poisson summation formula, which follows by repeated integration by parts for $|h| > H$.
\begin{lemma}[Truncated Poisson summation formula] \label{le:poisson} Let $\delta,\eta >0$ and  let $f$ be a smooth function supported in $[1,2]$ with $f^{(J)} \ll_J\delta^{-J}$ for all $J\geq 0$. Let $N> 1$ and $q \in \Z_{>0}$. Let $Z > 1$ and let
\begin{align*}
  H \geq  N^\eta \delta^{-1} q/N.  
\end{align*}
Then for any $C >0$
\begin{align*}
\sum_{n \equiv a \, (q)} f(\tfrac{n}{N}) = \frac{N}{q} \widehat{f}(0) + \frac{N}{q} \sum_{1 \leq |h| \leq H} \widehat{f} ( \tfrac{h N}{q}) e_q (ah) + O_{\eta,C}(N^{-C}),
\end{align*}
where $\widehat{f}(h):= \int_\R f(u)e(-hu) du$ is the Fourier transform.
\end{lemma}

The following lemma considers a smoothed sum over $\lambda_{\chi \xi^{\ell} }(n)$ and shows that, if the length of the sum is longer than the conductor, then we get almost perfect cancellation.
\begin{lemma} \label{le:smooothlambda}
 Let $\delta >0$ and  let $f$ be a smooth function supported in $[1,2]$ with $f^{(J)} \ll_J\delta^{-J}$ for all $J\geq 0$. Let $\lambda_{\chi \xi^{\ell} }(n)$ be eigenvalues \eqref{eq:hecekdef} for a fundamental discriminant $d > 0$ with $\chi \xi^{\ell} \neq 1$ and let $q \in \Z_{>0}$. Suppose that for some $\eta > 0$ we have $N> N^\eta  q d(1+\delta^{-2}+\ell^2).$ Then for any $C>0$
 \begin{align*}
     \sum_n f(\tfrac{n}{N})\lambda_{\chi\xi^\ell}(q n) \ll_{\eta,C} q^{o(1)} N^{-C}.
 \end{align*}
\end{lemma}
\begin{proof}
 Let $q=q_0 q_1$ with $q_0=\gcd(q,(2d)^\infty)$. Then we have the Hecke relation
 \begin{align*}
     \lambda_{\chi\xi^\ell}(q n)  = \lambda_{\chi\xi^\ell}(q_0) \sum_{r|\gcd (q_1,n)}\mu(r) (\tfrac{d}{r})\lambda_{\chi\xi^\ell}(\tfrac{q_1}{r}) \lambda_{\chi\xi^\ell}(\tfrac{n}{r}),
 \end{align*}
 which reduces the proof to the case $q=1$. By Mellin inversion we have
    \begin{align*}
         \sum_n f(\tfrac{n}{N})\lambda_{\chi\xi^\ell}(n) = \frac{1}{2 \pi i} \int_{(2)} N^{s} \widetilde{f}(s) L(s,\chi\xi^\ell) \d s,
    \end{align*}
    where the Mellin transform satisfies for any $J > 0$ by repeated integration by parts
    \begin{align*}
         \widetilde{f}(s) \ll_{J} (1+\delta |t|)^{-J}.
    \end{align*}
 Then by shifting the contour to $(-\sigma)$  using the bound \eqref{eq:Lfuncbound} we have
    \begin{align*}
         \sum_n f(\tfrac{n}{N})\lambda_{\chi\xi^\ell}(n) &\ll_{J} N^{-\sigma} \int_{\R} (1+\delta |t|)^{-J} (d (1+t^2+\ell^2))^{1/2+\sigma} \d t \ll_{\eta,C}  N^{-C}
    \end{align*}
 if we let $J>2\sigma+3$ and take $\sigma$ sufficiently large in terms of $C$ and $\eta$.
\end{proof}
We also need the following lemma for a smoothed sum over $\lambda_{1}^\flat(n,T)$, which gives a non-trivial bound as soon as $T$ is a bit larger than $q d^{1/2}$.
\begin{lemma} \label{le:lambdaflatsmooth}
     Let $\delta >0$ and  let $f$ be a smooth function supported in $[1,2]$ with $f^{(J)} \ll_J\delta^{-J}$ for all $J\geq 0$. Let $\lambda_{1}^\flat(n,T)$ be as in \eqref{eq:heckeflatdef} for a fundamental discriminant $d > 1$. Let $N > 1$ and let $q \in \Z_{>0}$. Then for any $\nu \in (0,\delta)$ we have
     \begin{align*}
           \sum_n f(\tfrac{n}{N})\lambda^\flat_{\chi\xi^\ell}(q n,T)  \ll \nu^{-1}\frac{q N}{T}d^{1/2} + \nu N^{1+o(1)}.
     \end{align*}
\end{lemma}
\begin{proof}
By Dirichlet divisor switching we get
    \begin{align*}
       \lambda^\flat_{\chi\xi^\ell}(m,T)  = \sum_{\substack{c|m \\ c > T}} (\tfrac{d}{c}) =  \sum_{\substack{c|m \\ c \leq m/T}}  (\tfrac{d}{m/c}).
    \end{align*}
    Approximating $c \leq m/T$ by a smooth function $f_\nu(\tfrac{m}{cT})$ with  $f_\nu^{(J)} \ll_J\nu^{-J}$, we get an error term $\nu N^{1+o(1)}$ by the divisor bound. It then suffices to show that for any $c_1 = \frac{c}{\gcd(c,q)} \leq 2qN/T$  we have
    \begin{align*}
         \sum_{n \equiv 0 \pmod{c_1}} f(\tfrac{n}{N})f_\nu(\tfrac{qn}{cT}) (\tfrac{d}{n/c_1}) =   \sum_{n } f(\tfrac{nc_1}{N})f_\nu(\tfrac{qnc_1}{cT}) (\tfrac{d}{n})  \ll \nu^{-1} d^{1/2}.
    \end{align*}
    This is a variant of the P\'olya-Vinogradov bound, and it follows by Poisson summation (Lemma \ref{le:poisson}) and the bound $d^{1/2}$ for the resulting Gauss sums.
\end{proof}

We need the following lemma, which gives a cheap but flexible version of the fundamental lemma of the sieve along arithmetic progressions.
\begin{lemma}  \label{le:cheapFLsieve}
Let $\delta >0$ and  let $f$ be a smooth function supported in $[1,2]$ with $f^{(J)} \ll_J\delta^{-J}$ for all $J\geq 0$. Let $N,X>1$ and let $q \in \Z_{>0}$. Suppose that for some small $\eta > 0$
\begin{align*}
    X^{\eta}\max\{ 1,\delta^{-1} q \} < N < X^{1/\eta}.
\end{align*}
Let $W:= X^{(\log\log X)^{-2}}$, $W_1:= X^{\eta^3}$, and define the normalized sieve weights 
    \begin{align*}
  \lambda^W_d &:= \frac{P(W)}{ \varphi(P(W))} \mu(d) \mathbf{1}_{d \leq W_1} \mathbf{1}_{d | P(W)}, \quad \quad
  \theta_n^W := \sum_{d| n} \lambda^W_d 
\end{align*}
Then for any $a \in \Z$ and for any $C>0$
\begin{align*}
    \sum_{n \equiv a \pmod{q}} f(\tfrac{n}{N})  \theta_n^W = &\frac{\mathbf{1}_{\gcd(a,q)=1}}{\varphi(q)} N \widehat{f}(0)    + 
 O_{\eta,C}\Bigl(\frac{N}{q}\widehat{f}(0) \Bigl( \frac{\gcd(a,q,P(W))^{1/10} }{(\log X)^{C}} +\frac{\gcd(a,q)^{1/10} }{W^{1/10}}\Bigr) \Bigr).  
\end{align*}
\end{lemma}
\begin{proof}
    By Lemma \ref{le:poisson} we have
    \begin{align*}
         \sum_{n \equiv a \pmod{q}} f(\tfrac{n}{N})  \theta_n^W =  N \widehat{f}(0) \frac{1}{q}\sum_{\substack{d \\ \gcd(d,q) | \gcd(a,q) }} \lambda_d^{W} \frac{\gcd(d,q)}{d}  + O_\eta(X^{-100}).
    \end{align*}
Here
    \begin{align} \label{eq:twosumsW}
    \begin{split}
        \sum_{\substack{d \\ \gcd(d,q)|\gcd(a,q) }} \lambda_d^{W} \frac{\gcd(d,q)}{d} = &\frac{P(W)}{\varphi(W)}\sum_{\substack{d |P(W) \\ \gcd(d,q)|\gcd(a,q) }}  \mu(d) \frac{\gcd(d,q)}{d} \\
 &+ O \bigg(\frac{P(W)}{\varphi(W)}\sum_{\substack{d|P(W)   \\ \gcd(d,q) | \gcd(a,q) \\d > W_1}}\frac{\gcd(d,q)}{d} \bigg).
    \end{split}
    \end{align}
   The first term in \eqref{eq:twosumsW} vanishes unless $\gcd(a,q,P(W))=1$, in which case it is equal to
    \begin{align*}
        \frac{P(W)}{\varphi(W)}\sum_{\substack{d|P(W)  \\ (d,q)=1}} \mu(d) \frac{1}{d}  = \frac{\gcd(q,P(W))}{\varphi(\gcd(q,P(W)))} = \frac{q}{\varphi(q)} (1+ O_C((\log X)^{-C}),
    \end{align*}
which gives the main term. Note that $\mathbf{1}_{\gcd(a,q,P(W))=1} =\mathbf{1}_{\gcd(a,q)=1} $ unless $\gcd(a,q) \geq W$, in which case we can use the upper bound $\mathbf{1}_{\gcd(a,q) \geq W} \leq W^{-1/10}\gcd(a,q)^{1/10}$ to absorb the main term to the error term. 

For $\gcd(a,q,P(W)) \geq W$ the second term in \eqref{eq:twosumsW} is bounded trivially by 
\begin{align*}
    \mathbf{1}_{\gcd(a,q,P(W)) \geq W} d(\gcd(a,q,P(W))(\log X)^2 \leq  W^{-1/20} (\log X)^2 \gcd(a,q,P(W))^{1/10}.
\end{align*} For $\gcd(a,q,P(W,P(W)))< W$ we have $e=\gcd(d,q) < W$, and the second term in \eqref{eq:twosumsW} is bounded by a bound for large smooth numbers (for instance, \cite[Lemma 9]{merikoskipolyprimes})
    \begin{align*}
       \ll  (\log X) \sum_{\substack{e|\gcd(a,q) \\e \leq W}} e
 \sum_{\substack{d|P(W) \\ e|d\\ d > W_1}}\frac{1}{d} \ll_C d(\gcd(a,q,P(W))) (\log X)^{-C}.
    \end{align*}
\end{proof}
When applying Lemma \ref{le:cheapFLsieve},
we bound the error term  using the following simple estimates, which hold for any $P_0\in\{0,P(W)\}$ and $D,Y,n_1,n_2\in (1,X]$ with $\gcd(n_1,n_2)=1$
\begin{align} \label{eq:gcdbound1}
\begin{split}
      \sum_{d_1 \sim D}  \sum_{y_1,y_2 \sim Y} \sum_{\substack{x_1,x_2,u \pmod{d_1y_1^3  n_2 } \\  b n_1 u_2 y_2^3 \equiv a (n_2 x_1^2 - n_1 x_2^2)\pmod{d_1 y_1^3 n_2} }} \frac{\gcd(n_2 x_1^2-n_1 x_2^2,d_1 y_1^3 n_2,P_0)^{1/10}}{( d_1 y_1^3 n_2)^2}   \\
    \ll  D Y^2 (\log X)^{O(1)}   d(\gcd(n_2,P_0))^{O(1)}, \hspace{-50pt}
\end{split} \\ \label{eq:gcdbound2}
\begin{split}
     \sum_{c \sim D}  \sum_{y_1,y_2 \sim Y} \sum_{\substack{u_1,u_2 \pmod{c} \\ b n_2 y_1^3 u_1  \equiv b n_1 y_2^3 u_2  \pmod{c} }} \frac{\gcd(y_1^3 u_1,c)^{1/10}\gcd(y_2^3 u_2,c)^{1/10}}{c}  
    \ll  D Y^2  (\log X)^{O(1)} .
\end{split}
\end{align} 
In the proofs we will also need the divisor bound for rough divisors of $n \leq X$ with $W$ as in Lemma \ref{le:cheapFLsieve}
\begin{align} \label{eq:roughdivisorbound}
    \#\{d| n: \, \gcd(d,P(W))=1 \} \leq  2^{\log n / \log W} \leq W^{o(1)}.
\end{align}

We will need the following truncated approximation to $L(1, (\tfrac{4d}{\cdot}))^{-1}$ on average over the moduli $d=n_1 n_2$ for the proof of Proposition \ref{prop:typeII}.
 \begin{lemma} \label{le:Linverse}
     We have  for any $K,D \geq 1$
     \begin{align*}
         \sum_{\substack{d \leq D \\ \text{\emph{square-free and odd }}}} \bigg| \frac{1}{L(1, (\tfrac{4d}{\cdot}))} - \sum_{k \leq K} \frac{\mu(k)(\tfrac{4d}{k})}{k} \bigg|^2 \ll \frac{D^{1+o(1)} }{K} + D^{o(1)}.
     \end{align*}
 \end{lemma}
\begin{proof}
   Plugging in
    \begin{align*}
        \frac{1}{L(1, (\tfrac{4d}{\cdot}))} = \sum_{k} \frac{\mu(k)(\tfrac{4d}{k})}{k},
    \end{align*}
  applying quadratic reciprocity, a dyadic partition, and Cauchy-schwarz, it suffices to bound
  \begin{align*}
      \sum_{K_1=2^j \geq K/2} (\log K_1)^2 \sum_{\substack{d \leq D \\ \text{square-free and odd}}} \bigg|\sum_{k \sim K_1} \frac{\mu(k)(\tfrac{k}{d})}{k} \bigg|^2.
  \end{align*}
We split into three cases, $K_1 \in [K/2,D^3]$, $K_1 \in (D^3,e^{D^\eps}]$, and $K_1 > e^{D^{\eps}}$.

By the Siegel-Walfisz theorem terms where $K_1 > e^{D^\eps}$ contribute
\begin{align*}
    \ll_C  \sum_{K_1=2^j > e^{D^\eps}} \frac{D}{(\log K_1)^C}\ll_\eps 1
\end{align*}
by taking $C$ sufficiently large in terms of $\eps$.

For $K_1 \leq D^3$ applying Heath-Brown's large sieve for quadratic characters \cite[Theorem 1]{hbquad} we get
    \begin{align*}
        \sum_{K/2 \leq K_1=2^j \leq D^3 } (\log K_1)^2 \sum_{\substack{d \leq D \\ \text{square-free and odd}}} \bigg|\sum_{k \sim K_1} \frac{\mu(k)(\tfrac{k}{d})}{k} \bigg|^2 &\ll    \sum_{K/2 \leq K_1=2^j \leq D^3 }(DK_1)^{1+o(1)} (\frac{D}{K_1} + 1)  \\
        &\ll \frac{D^{1+o(1)} }{K} + D^{o(1)}.
    \end{align*}
  
For $D^3 \leq K_1 \leq e^{D^{\eps}}$, it suffices to show that for any $\beta_d$ we have
\begin{align} \label{eq:dualclaim}
   \sum_{k \sim K_1}  \bigg|  \sum_{\substack{d \leq D \\ \text{square-free and odd}}} \beta_d (\tfrac{k}{d}) \bigg|^2 \ll  K_1 \|\beta\|_2^2,
\end{align}
since  then by the duality principle (see \cite[Chapter 7.1]{IKbook}, for instance)
\begin{align*}
     \sum_{D^3 < K_1=2^j \leq e^{D^{\eps}} } (\log K_1)^2 \sum_{\substack{d \leq D \\ \text{square-free and odd}}} \bigg|\sum_{k \sim K_1} \frac{\mu(k)(\tfrac{k}{d})}{k} \bigg|^2 &\ll       \sum_{D^3 < K_1=2^j \leq e^{D^{\eps}} } (\log K_1)^2 \ll_\eps D^{3\eps}.
\end{align*}
To show \eqref{eq:dualclaim}, inserting a smooth majorant $F$ for $[1,2]$ and expanding the square the left-hand side is bounded by
\begin{align*}
  \ll    \sum_{\substack{d_1,d_2 \leq D  \\ \text{square-free and odd} } }  \beta_{d_1} \overline{\beta_{d_2}} \sum_{k} F(\tfrac{k}{K_1})  (\tfrac{k}{d_1 d_2})  \ll K_1 \sum_{\substack{d \leq D  \\ \text{square-free and odd}  } }  |\beta_{d}|^2,
\end{align*}
since  by Poisson summation (Lemma \ref{le:poisson}) 
\begin{align*}
    \sum_{k} F(\tfrac{k}{K_1})  (\tfrac{k}{d_1 d_2})  \ll D^{-100} + \frac{K_1}{d_1d_2}\sum_{ k \pmod{d_1d_2}}  (\tfrac{k}{d_1 d_2}) \ll  D^{-100} +K_1 \mathbf{1}_{d_1=d_2}.
\end{align*}
\end{proof}

\section{Set-up and arithmetic information}
We let $\delta = \delta(X):= (\log X)^{-c}$ for some fixed large $c >0$. Let $f,f_1,f_2$ denote non-negative non-zero smooth functions supported in $[1,1+\delta]$ and satisfying the derivative bounds $ f^{(J)}, f_1^{(J)}, f_2^{(J)}\ll_J \delta^{-J}$ for all $J \geq 0$.  For  $A \in (\delta X^{1/2}, X^{1/2}]$ and $B\in (\delta X^{1/3}, X^{1/3}]$ we define the sequences $\mathcal{A}=(a_n)$, $\mathcal{B}=(b_n)$, and their difference $\mathcal{W}=(w_n)$ by 
\begin{align*}
    a_n &:= a_n(a,b,f_1,f_2,A,B) = \sum_{n=a x^2+b y^3} f_1(\tfrac{x}{A}) f_2(\tfrac{y}{B}), \\
    b_n &:=  b_n(a,b,f,f_1,f_2,A,B,W)=  f(\tfrac{n}{X}) \frac{AB \widehat{f_1}(0)\widehat{f_2}(0)  }{X\widehat{f}(0)  } \\
    w_n &:= a_n-b_n
\end{align*}
Theorem \ref{thm:asymp}  is an immediate corollary, via finer-than-dyadic decomposition and the Prime number theorem, of the following smoothed version which is proved in Section \ref{sec:sieve}.
\begin{theorem} \label{thm:technicalasymp}
 Let $a,b >0$ be coprime integers. Let $\eps > 0$ and suppose that Conjecture $\mathrm{C}_a(\eps)$ holds. Let $A \in (\delta X^{1/2}, X^{1/2}]$ and $B\in (\delta X^{1/3}, X^{1/3}]$. Then 
    \begin{align*}
        \sum_{n} \Lambda(n) w_n &\ll \eps AB \widehat{f_1}(0)\widehat{f_2}(0).
    \end{align*}
\end{theorem}
For the proof, we require three types of arithmetic information. Proposition \ref{prop:typeI} is proved in Section \ref{sec:typei} and Proposition \ref{prop:typeII} is proved in Sections \ref{sec:typeii} and \ref{sec:typeii1}.  Conjecture $\mathrm{C}_a(\eps)$ is required only for Proposition \ref{prop:typeII}, the other results are unconditional. The proof of Proposition \ref{prop:typeI2} will appear in \cite{GMquadratic}, where it can be done more economically as a corollary of more general considerations, via the spectral methods of $\operatorname{GL}(2)$ automorphic forms. Only for Proposition \ref{prop:typeI2} do we need that $a,b$ are positive, and it may be possible to relax this assumption using ideas from \cite{toth}. 
\begin{proposition}[Type I information up to $5/9$] \label{prop:typeI}
 Let $a,b \neq 0 $ be coprime integers. For any square-free $d \leq X^{5/9-\eta}$  we have
    \begin{align*}
        \sum_{n } w_{dn}  \ll  d^{-1} X^{5/6-\eta}. 
    \end{align*}
\end{proposition}

\begin{proposition}[Type II information in $(1/6,1/3)$]  \label{prop:typeII}
 Let $a,b \neq 0 $ be coprime integers.  Let $\eps > 0$ and assume that Conjecture $\mathrm{C}_a(\eps)$  holds.  Let  $M,N>1$ satisfy for some $\eta>0$
 \begin{align*}
     X^{1/6+\eta} \leq N \leq X^{1/3-2\eps/3-\eta}, \quad MN= X.
 \end{align*}
Let $W:= X^{(\log\log X)^{-2}}$. Let $\alpha_m, \beta_n$ be bounded coefficients supported on square-free integers with $\gcd(mn,P(W))=1$. Suppose that $\beta_n$ satisfies the Siegel-Walfisz condition, that is,  for all  $r,q,N' \leq 2 N$ we have for any $C>0$
 \begin{align} \label{eq:betacondition}
     \sum_{\substack{n \leq N' \\ n \equiv r \pmod{q}}} \beta_n = \mathbf{1}_{\gcd(r,q)=1}\frac{N'}{N \varphi(q)} \sum_{\substack{n \leq N \\ (n,q)=1}} \beta_n + O_C(N(\log N)^{-C}).
 \end{align}
Then we have for any $C>0$
 \begin{align*}
     \sum_{m \sim M} \sum_{n} \alpha_m \beta_n w_{mn} \ll_C \frac{  X^{5/6}}{(\log X)^C}.
 \end{align*}
\end{proposition}

\begin{proposition}[Type I$_2$ information up to $1/4$] \label{prop:typeI2}
 Let $a,b > 0 $ be coprime integers. For any $K \leq X^{3/4}$  we have
  \begin{align*}
    \sum_{d \leq X^{1/4-\eta}} \bigg| \sum_{k \equiv 0\pmod{d} }  \mu^2(k) f(\tfrac{k}{K}) 
 \sum_{n }  w_{k n} \bigg| \ll d^{-1} X^{5/6-\eta}.
  \end{align*}
\end{proposition}

\section{Proof of Proposition \ref{prop:typeI}} \label{sec:typei}
The argument is a routine application of the Poisson summation formula and the Weil bound for exponential sums. We have by Lemma \ref{le:poisson}
\begin{align*}
      & \sum_{n } a_{dn}  =\sum_{\substack{x,y \\ ax^2+by^3 \equiv 0 \pmod{d}}}  f_1(\tfrac{x}{A}) f_2(\tfrac{y}{B})  = M_\mathcal{A} + O(X^{2\eta^2} E_\mathcal{A} ) + O_\eps(X^{-100}),
\end{align*}
where for $H_1=X^{\eta^2} d /A,H_2=X^{\eta^2} d /B$
\begin{align*}
 M_\mathcal{A}  &=  AB \widehat{f_1}(0) \widehat{f_2}(0) \frac{\#\{x,y \in \Z/d\Z: ax^2+by^3 \equiv 0 \pmod{d}\}}{d^2}  \\
 E_\mathcal{A}  &= \frac{1}{H_1 H_2} \sum_{\substack{|h_1| \leq H_1\\ |h_2| \leq H_2 \\ (h_1,h_2) \neq (0,0)}} \bigg| \sum_{\substack{x,y \pmod{d} \\ ax^2+by^3 \equiv 0 \pmod{d} }} e_d(h_1 x + h_2 y) \bigg|
\end{align*}
Here by the Chinese remainder Theorem
\begin{align*}
   N_{a,b}(d)=\#\{x,y \in \Z/d\Z: ax^2+by^3 \equiv 0 \pmod{d}\} = \prod_{p} N_{a,b}(p).
\end{align*}
For $(a,p)=(b,p)=1$ substituting $x=zy$ gives $ N_{a,b}(p) = p.$ On the other hand, using $(a,b)=1$ we note that $p|a$ implies $p|y$ and $p|b$ implies $p|x$, so that also for $p|ab$ we  have $ N_{a,b}(p) = p.$  Therefore, we get
\begin{align*}
    M_\mathcal{A}  =  \frac{AB \widehat{f_1}(0) \widehat{f_2}(0) }{d}.
\end{align*}
We also have by Lemma \ref{le:poisson} 
\begin{align*}
  \sum_{n } b_{dn}  = \frac{AB \widehat{f_1}(0)\widehat{f_2}(0)  }{X\widehat{f}(0)  }     \sum_{n \equiv 0 \pmod{d}}   f(n/X) = \frac{AB \widehat{f_1}(0) \widehat{f_2}(0) }{d}  + O(X^{-100}),
\end{align*}
so that the two main terms cancel precisely. It then remains to bound the error term $E_\mathcal{A} $. We have by the Chinese remainder theorem 
\begin{align*}
 S_d(h_1,h_2) =   \sum_{\substack{x,y \pmod{d} \\ ax^2+by^3 \equiv 0 \pmod{d} }} e_d(h_1 x + h_2 y) = \prod_{p|d}  S_p(h_1 \overline{d/p},h_2\overline{d/p}) 
\end{align*}
For $(a,p)=(b,p)=1$ the substitution $x \equiv zy$ gives $y=-a\overline{b}z^2$, so that by the Weil bound (see, for instance, \cite{bombieri}) 
\begin{align*}
    S_p(h_1,h_2)   = \sum_{\substack{z \in \Z/p\Z }} e_p(-a \overline{b} (h_1 z^3 + h_2 z^2))  \ll \gcd(h_1,h_2,p)^{1/2} p^{1/2}.
\end{align*}
Thus, $ S_d(h_1,h_2) \ll \gcd(h_1,h_2,d)^{1/2} d^{1/2+o(1)}, $ and for $d \leq X^{5/9-\eta}$  
\begin{align*}
     E_\mathcal{A}  & \ll  d^{1/2+o(1)} \frac{1}{H_1 H_2} \sum_{\substack{|h_1| \leq H_1\\ |h_2| \leq H_2 \\ (h_1,h_2) \neq (0,0)}} \gcd(h_1,h_2,d)^{1/2}  \\
     &\ll d^{1/2+o(1)} \ll d^{-1}X^{5/6-3\eta/2+o(1)}.
\end{align*}
\qed

\section{Proof of Proposition \ref{prop:typeII}: initial reductions} \label{sec:typeii}

We wish to apply Cauchy-Schwarz to $m$ but face a problem, namely, the distribution of $a_n$ does not match $b_n$ modulo squares or larger powers. We have two options, either to remember that $m$ is square-free or modify the sequence $b_n$. Both options lead to unfortunate complications, but the latter allows us to minimize the conjecture required since for the first option we would also need to consider twists of $\chi \xi^\ell$ by Dirichlet characters. 

A robust solution is to construct an auxiliary sequence $a_n^{(2)}$ in between $a_n$ and $b_n$ as a kind of random model for $a_n$. While this causes some bother during the initial stages of the argument, it will greatly simplify the endgame. Let $F_{\eta^2}(u)$ denote a smooth non-negative bump function supported on $u \in [1,1+X^{-\eta^2}]$ with $F_{\eta^2}^{(J)} \ll_J X^{J \eta^2}$ and denote
\begin{align*}
    F_2(\tfrac{u}{U}) := \frac{F_{\eta^2}(\tfrac{u}{U})}{U^{2/3}\widehat{F_{\eta^2}}(0)}  \quad \text{with} \quad  U:=X^{1-\eta^2}.
\end{align*}
Let $ \lambda^W_d, \theta_n^W = \sum_{d|n}\lambda_d^W$ be as in Lemma \ref{le:cheapFLsieve}.  We define 
\begin{align*}
    \gamma^{(1)}_v &:= \sum_{y}\mathbf{1}_{v=b y^3} f_2(\tfrac{ y}{B}), \quad \quad \quad
    \gamma^{(2)}_v := \sum_{v=buy^3}    f_2(\tfrac{U^{1/3} y}{B})      F_2(\tfrac{u}{U}) \theta_u^W,
\end{align*}
 and for $j \in \{1,2\}$
\begin{align*}
    a_n^{(j)} := \sum_{n=a x^2+ v}  f_1(\tfrac{x}{A}) \gamma^{(j)}_v. 
\end{align*}
Note that then $a^{(1)}_n = a_n$. The function $F_2$ is normalized so that the densities match, that is, $    \int F_2(\tfrac{u}{U}) f_2(\tfrac{U^{1/3} y}{ B}) \d u \d y =   \int f_2(\tfrac{ y}{ B}) \d y $.
To show Proposition \ref{prop:typeII}, it suffices to prove the following two variants of Proposition \ref{prop:typeII}.
\begin{proposition} \label{prop:typeii0}
     Suppose that the assumptions of Proposition \ref{prop:typeII} hold.  Then for any $C>0$ 
     \begin{align*}
           \sum_{m\sim M} \alpha_m  \sum_n \beta_n (a^{(2)}_{m n}-b_{mn}) \ll_C \frac{X^{5/6}}{(\log X)^C}
     \end{align*}
\end{proposition}
\begin{proposition} \label{prop:typeii1}
  Suppose that the assumptions of Proposition \ref{prop:typeII} hold.  Then for any $C>0$ 
 \begin{align*}
     \sum_{m \sim M} \bigg|\sum_{n} \beta_n (a^{(1)}_{mn} -a^{(2)}_{mn} ) \bigg| \ll_C \frac{  X^{5/6}}{(\log X)^C}.
 \end{align*}
\end{proposition}
The Proposition \ref{prop:typeii0} is easy to prove and we do this immediately. We will prove Proposition \ref{prop:typeii1}  in Section \ref{sec:typeii1}.
\subsection{Proof of Proposition \ref{prop:typeii0}}
By the Siegel-Walfisz property \eqref{eq:betacondition} and summation by parts we have (denoting $N=X/M$)
\begin{align*}
    \sum_{m\sim M} \alpha_m \sum_{n} \beta_n b_{mn} = AB \widehat{f_1}(0)\widehat{f_2}(0)  \sum_{m\sim M} \frac{\alpha_m}{m} \frac{1}{N} \sum_{n \leq N} \beta_n + O_C(X^{5/6} (\log X)^{-C}).
\end{align*}
Substituting for the variable $u$ and using the Poisson summation formula (Lemma \ref{le:poisson}) to the variable $x$, we have
\begin{align*}
    a^{(2)}_{mn} = &\sum_{d}  \lambda^W_d\sum_{y}f_2(\tfrac{U^{1/3} y}{B})\sum_{\substack{x \\ ax^2  \equiv n \pmod{b d y^3}}} f_1(\tfrac{x}{A})  F_2(\tfrac{mn-ax^2}{by^3 U})   \\
    =&\sum_{d } \lambda^W_d\sum_{y} \frac{\varrho_2(mn,b d y^3)}{b d y^3}  \mathcal{I}(mn,y) + O(X^{-100}), \\
    \varrho_2(a,q) :=& \#\{x \in \Z/q: ax^2  \equiv n \pmod{q}\},\\
    \mathcal{I}(n,y) = &f_2(\tfrac{U^{1/3} y}{B})\int f_1(\tfrac{x}{A}) F_2(\tfrac{n-ax^2}{by^3 U})  \d x = \frac{by^3}{\sqrt{a}}f_2(\tfrac{U^{1/3} y}{B})\int f_1(\tfrac{\sqrt{n-by^3 u}}{A\sqrt{a}}) F_2(\tfrac{u}{U})  \frac{\d u}{\sqrt{n-by^3 u}}.
\end{align*}
By differentiation under integration in the latter expression, we have for all $J_1,J_2 \geq 0$ 
\begin{align*}
    \partial_{n}^{J_1} \partial_{y}^{J_2} \mathcal{I}(n,y) \ll_{J_1,J_2} \frac{B^3}{A U^{2/3}}  \delta^{-J_1-J_2} n^{-J_1} y^{-J_2}.
\end{align*}
Since $\gcd(mn,P(W))=1$, the contribution from $\gcd(mn,q y^3)>1$ is negligible by crude estimates. Therefore, using summation by parts on $n$ and gluing the variables $q=bd$, the task is reduced to showing that for any $N' \leq 2N$ and $Y= B U^{-1/3} \leq X^\eta$ we have
\begin{align*}
    \sum_{q \leq X^\eta} \sum_{y \sim Y} \frac{1}{q y} \sum_{\substack{m \sim M\\ \gcd(mn,q y^3)=1}} \bigg| \sum_{\substack{n \leq N' \\\gcd(mn,q y^3)=1}} \beta_n \Bigl(\varrho_2(mn,q y^3) -   1\Bigr)  \bigg| \ll_C \frac{MN}{(\log X)^C}.
\end{align*}
Expanding the defininition of $\varrho_2(mn,qy^3)$ we have $\gcd(x,qy^3)=1$. Pulling the sum over $x \in (\Z/qy^3\Z)^\times$ outside and taking the maximum, we need to show that
\begin{align*}
    \sum_{q \leq X^\eta} \sum_{y \sim Y}  \max_{\substack{x }} \bigg| \sum_{\substack{n \leq N' \\ \gcd(n,qy^3)=1  }} \beta_n \bigg( \mathbf{1}_{n \equiv ax^2 \pmod{q y^3}}-   \frac{1}{ \varphi(q y^3)} \bigg) \bigg| \ll_C \frac{N}{(\log X)^C}.
\end{align*}
Expanding the congruence into Dirichlet characters modulo $q y^3$, the main terms for the principal character cancel. It then remains to show that the error term satisfies
\begin{align*}
 \sum_{q \leq X^\eta} \sum_{y \sim Y} \frac{1}{\varphi(q y^3)}\sum_{\substack{\chi \pmod{q y^3} \\ \chi \neq  \chi_0}} \bigg| \sum_{n \leq N'} \beta_n \chi(n)  \bigg|  \ll_C \frac{N}{(\log X)^C}
\end{align*}
Sorting by the conductor of the character $\chi$, we apply the Siegel-Walfisz property \eqref{eq:betacondition} for the small conductors $\leq (\log X)^{C_1}$. For the large conductors   $> (\log X)^{C_1}$ we may apply Cauchy-Schwarz,  estimate crudely the sum over cubes  by $\sum_{y \sim Y} S(y^3) \leq \sum_{z \sim Y^3} S(z)$,  and use the classical large sieve for multiplicative characters \cite[(9.52)]{odc} to get the claim, once $C_1$ is sufficiently large in terms of $C.$  \qed

\subsection{Lemmas on additive convolutions of $\gamma^{(i)}_v$}
Define for $i_1,i_2 \in \{1,2\}$
\begin{align} \label{eq:upsilondef}
     \Upsilon^{i_1,i_2}_{n_1,n_2}(m)  &:= \sum_{n_1 v_2-n_2v_1 = a m} \gamma^{(i_1)}_{v_1} \gamma^{(i_2)}_{v_2}
\end{align}
To prepare for the proof of Proposition 
\ref{prop:typeii1}, we need the following two lemmas.
\begin{lemma} \label{le:22smooth}
For  $n_1,n_2\leq X^{1/3-\eta}$ with  $\gcd(n_1,n_2)=1$  we have
    \begin{align*}
          \Upsilon^{2,2}_{n_1,n_2}(m) 
         &= \frac{B^2}{N_1X} \sum_{e \leq X^{4\eta^2}} \mathbf{1}_{b e|m} G_{e,n_1,n_2}( \tfrac{m}{n_1 X}) + O_\eta(X^{-100}),
    \end{align*}
    where $G_{e,n_1,n_2}(u)$ is a smooth function supported on $|u| \leq  4N_1 B^3$ which satisfies for all $J\geq 0$ the derivative bound  $G^{(J)}(u) \ll_J X^{J\eta^2 + O(\eta^2)} $.
\end{lemma}
\begin{proof}
    We have
      \begin{align*}
            \Upsilon^{2,2}_{n_1,n_2}(m)  &=  \sum_{d_1,d_2} \lambda_{d_1}^W \lambda_{d_2}^W\sum_{\substack{y_1,y_2,u_1 \\ b(n_1 d_2 u_2 y_2^3 - n_2 d u_1 y_1^3) = am }} \frac{ f_2(\tfrac{U^{1/3} y_1}{B})F(\tfrac{d_1 u_1}{U})}{U^{2/3}\widehat{F}(0)} \frac{ f_2(\tfrac{U^{1/3} y_2}{B})F(\tfrac{d_2 u_2}{U})}{U^{2/3}\widehat{F}(0)}.
    \end{align*}
    Substituting 
    \begin{align*}
    u_2  := \frac{a\tfrac{m}{b} +n_2 d_1 u_1 y_1^3 }{d_2 n_1 y_2^3}
    \end{align*}  
   and applying Poisson summation to $u_1$ with $n_2 d_1 u_1 y_1^3 \equiv a\tfrac{m}{b} \pmod{d_2 n_1 y_2^3}$   gives the claim after splitting into the parts depending on
    \begin{align*}
        e_0 := \gcd(d_2n_1  y_2^3, d_1 n_2 y_1^3) |  \, a\tfrac{m}{b}
    \end{align*}
 and letting $e=\tfrac{e_0}{\gcd(e_0,a)}$. Note that by $\gcd(n_1,n_2) =1$ we have $e \leq d_1 d_2 y_1^3y_2^3 \leq X^{4\eta^2}$.
\end{proof}
\begin{lemma} \label{le:12vs22}
Let $n_1,n_2 \sim N_1 \leq X^{1/3-\eta}$ be square-free with $\gcd(n_1,n_2)=\gcd(n_1n_2,P(W))=1$, and let $(i_1,i_2) \in \{(1,2),(2,1)\}$. Then for any $C>0$
  \begin{align*}
      \sum_{x_1,x_2} f_1(\tfrac{x_1}{A})  f_1(\tfrac{x_2}{A})   \Bigl(\Upsilon_{n_1,n_2}^{i_1,i_2}(n_2 x_1^2 - n_1 x_2^2) - \Upsilon_{n_1,n_2}^{2,2}(n_2 x_1^2 - n_1 x_2^2) \Bigr) \ll_C  \frac{B^2}{N_1 (\log X)^{C}}. 
  \end{align*}
\end{lemma}
\begin{proof}
   Expanding the definition of $ \Upsilon_{n_1,n_2}^{2,1}$ we have
   \begin{align*}
        &  \sum_{x_1,x_2} f_1(\tfrac{x_1}{A})f_1(\tfrac{x_2}{A})   \Upsilon_{n_1,n_2}^{2,1}(n_2 x_1^2 - n_1 x_2^2) \\ &= \sum_{d_1} \lambda_{d_1}^W    \sum_{\substack{x_1,x_2,y_1,y_2,u_1 \\ a(n_2 x_1^2 - n_1 x_2^2)= b(n_1y_2^3-  n_2 d_1 u_1 y_1^3) }}  f_1(\tfrac{x_1}{A})f_2(\tfrac{U^{1/3} y_1}{B})F_2(\tfrac{d_1 u_1}{U})f_1(\tfrac{x_2}{A})f_2(\tfrac{y_2}{B}).
   \end{align*}
   Making the variable $u_1$ implicit by substitution we are summing over the congruence
   \begin{align*}
          b n_1 y_2^3 \equiv a (n_2 x_1^2 - n_1 x_2^2) \pmod{b d_1 y_1^3 n_2}.
   \end{align*}
   Since the modulus is $d_1 y_1^3 n_2 \leq X^{1/3-\eta/2},$ an application of Poisson summation formula (Lemma \ref{le:poisson}) to the variables $x_1,x_2,y_2$ produces a main term $\mathcal{M}^{2,1}_{n_1,n_2}$ with an error term $O(X^{-100})$.  Similarly, we have
   \begin{align*}
        &\sum_{x_1,x_2} f_1(\tfrac{x_1}{A}) f_1(\tfrac{x_2}{A})   \Upsilon_{n_1,n_2}^{2,2}(n_2 x_1^2 - n_1 x_2^2) \\ &= \sum_{d_1} \lambda_{d_1}^W    \sum_{\substack{x_1,x_2,y_1 ,y_2,u_1,u_2\\a(n_2 x_1^2 - n_1 x_2^2)= b(n_1 u_2 y_2^3-  n_2 d_1 u_1 y_1^3) }}  f_1(\tfrac{x_1}{A})  f_2(\tfrac{U^{1/3} y_1}{B})F_2(\tfrac{d_1 u_1}{U}) f_1(\tfrac{x_2}{A})f_2(\tfrac{U^{1/3} y_2}{B})F_2(\tfrac{ u_2}{U})\theta_{u_2}^W .
   \end{align*}
 Making the variable $u_1$ implicit by substitution,  we are summing over the congruence
   \begin{align*}
        b n_1 u_2 y_2^3 \equiv a (n_2 x_1^2 - n_1 x_2^2) \pmod{b d_1 y_1^3 n_2}..
   \end{align*}
We apply Poisson summation (Lemma \ref{le:poisson}) to $x_1,x_2$  and Lemma \ref{le:cheapFLsieve} to $u_2$, using  \eqref{eq:gcdbound1} to bound the error term. This produces a main term with the count
\begin{align*}
    \#\{(x_1,x_2,u_2)\in (\Z/d_1 y_1^3 n_2 \Z)^2\times(\Z/d_1 y_1^3 n_2 \Z)^\times:  b n_1 u_2 y_2^3 \equiv a (n_2 x_1^2 - n_1 x_2^2) \pmod{d_1 y_1^3 n_2} \}.
\end{align*}
As a function of $y_2$, this count depends only on $\gcd(y_2^3,d_1 y_1^3 n_2)$. The part where $\gcd(y_2^3,d_1 y_1^3 n_2) >X^{\eta^3}$ is negligible by crude estimates. For $\gcd(y_2^3,d_1 y_1^3 n_2) \leq X^{\eta^3}$ we may apply Poisson summation formula to $y_2$. This produces a main term which matches $\mathcal{M}^{2,1}_{n_1,n_2}$, since 
\begin{align*}
    \#&\{(x_1,x_2,y_2,u_2)\in (\Z/d_1 y_1^3 n_2 \Z)^3\times(\Z/d_1 y_1^3 n_2 \Z)^\times:  b n_1 u_2 y_2^3 \equiv a (n_2 x_1^2 - n_1 x_2^2) \pmod{d_1 y_1^3 n_2} \} \\
    &=   \varphi(d_1 y_1^3 n_2) \#\{(x_1,x_2,y_2)\in (\Z/d_1 y_1^3 n_2 \Z)^3:  b n_1  y_2^3 \equiv a (n_2 x_1^2 - n_1 x_2^2)\pmod{d_1 y_1^3 n_2} \}
\end{align*}
by making the change of variables $(x_1,x_2,y_2) \mapsto (u_2^2 x_1,u_2^2 x_2,u y_2)$.
\end{proof}
\section{Proof of Proposition \ref{prop:typeii1}} \label{sec:typeii1}
\subsection{Application of Cauchy-Schwarz}
Without loss of generality, we may assume that $\beta_n$ is real-valued and insert $n \sim N$ for $N \asymp X/M$. By applying Cauchy-Schwarz, dropping $m\sim M$, and expanding the dispersion, we have
\begin{align*}
      \sum_{m \sim M} \bigg|\sum_{n \sim N}  \beta_n(a^{(1)}_{mn} -a^{(2)}_{mn} ) \bigg| \ll M^{1/2}  \big( \mathcal{U}(1,1)-\mathcal{U}(1,2)-\mathcal{U}(2,1)+ \mathcal{U}(2,2) \big)^{1/2},
\end{align*}
where for $i_1,i_2 \in \{1,2\}$
\begin{align*}
\mathcal{U}(i_1,i_2) =& \sum_{n_1,n_2 \sim N} \beta_{n_1} \beta_{n_2} \sum_{m}  a^{(i_1)}_{mn_1}   a^{(i_2)}_{mn_2} 
= \sum_{n_0 \leq 2N}\sum_{\substack{n_1,n_2 \sim N/n_0 \\\gcd(n_1,n_2)=1 }} \beta_{n_0 n_1} \beta_{n_0 n_2} \mathcal{V}_{n_0}(n_1,n_2)
\end{align*}
with
\begin{align*}
 \mathcal{V}_{n_0}(n_1,n_2) =   \sum_{m}  \sum_{\substack{mn_0 n_1=a x_1^2+ v_1   \\ mn_0n_2 = ax_2^2+v_2 }}  f_1(\tfrac{x_1}{A}) f_1(\tfrac{x_2}{A})  \gamma_{v_1}^{(i_1)}  \gamma_{v_2}^{(i_2)} 
\end{align*}
It then suffices to show that for some $\mathcal{Y}$ we have for $i_1,i_2 \in \{1,2\}$
\begin{align} \label{eq:claimU}
\mathcal{U}(i_1,i_2)  =    \mathcal{Y} + O_C( N X^{2/3}(\log X)^{-C})   
\end{align}
We first bound the diagonal contribution where $n_1=n_2=1$.
\subsection{Contribution from the diagonal $n_1=n_2=1$ }
By using the divisor bound
\begin{align*}
   \sum_{n} |\beta_n|^2\mathcal{V}_n(1,1)
&\ll  X^{o(1)}\sum_{v_1,v_2} |\gamma_{v_1}^{(i_1)}  \gamma_{v_2}^{(i_2)}| \sum_{\substack{x_1,x_2 \\ a(x_1^2-x_2^2) = v_2 -v_1 }}f_1(\tfrac{x_1}{A}) f_1(\tfrac{x_2}{A})  \\
  &\ll  X^{o(1)} (A B + B^2) \ll N X^{2/3-\eta}
\end{align*}
by using $N \geq X^{1/6+\eta}$, where $x_1= \pm x_2$ contributed $\ll X^{o(1)} AB$ and $x_1\neq \pm x_2$ contributed $\ll X^{o(1)} B^2$ by a divisor bound for the number of representations as $x_1^2-x_2^2  = (x_1-x_2)(x_1+x_2)$.
\subsection{Contribution from $n_0=1$}
This will give the main term, we postpone bounding the contribution from the pseudo-diagonal terms $n_0 > 1$ (where $n_0 \geq W$ by $\gcd(n,P(W))=1$) to Section \ref{sec:n0contribution}. For $n_0=1$ we want to evaluate
\begin{align*}
    \mathcal{U}_1(i_1,i_2) 
=& \sum_{\substack{n_1,n_2 \sim N \\\gcd(n_1,n_2)=1}} \beta_{ n_1} \beta_{ n_2} \mathcal{V}_{1}(n_1,n_2)
\end{align*}
where
\begin{align*}
    \mathcal{V}_{1}(n_1,n_2) &=\sum_{m}  \sum_{\substack{m n_1=a x_1^2+ v_1   \\ mn_2 = ax_2^2+v_2}}  f_1(\tfrac{x_1}{A}) f_1(\tfrac{x_2}{A})  \gamma_{v_1}^{(i_1)}  \gamma_{v_2}^{(i_2)}  = \sum_{\substack{x_1,x_2,v_1,v_2 \\  n_2(ax_1^2+v_1 ) = n_1(ax_2^2+v_2)}} 
 f_1(\tfrac{x_1}{A})f_1(\tfrac{x_2}{A})\gamma_{v_1}^{(i_1)}  \gamma_{v_2}^{(i_2)} .
\end{align*}
We have $n_2 x_1^2 \neq n_1 x_2^2$ since $n_1,n_2$ are square-free and coprime. We rearrange the equation to  $ a (n_2 x_1^2-n_1 x_2^2) = n_1 v_2-n_2 v_1$
and obtain
\begin{align*}
      \mathcal{U}_1(i_1,i_2)  = \sum_{\substack{n_1,n_2 \sim N \\\gcd(n_1,n_2)=1}} \beta_{ n_1} \beta_{ n_2}   \sum_{x_1,x_2} f_1(\tfrac{x_1}{A})  f_1(\tfrac{x_2}{A})   \Upsilon_{n_1,n_2}^{i_1,i_2}(n_2 x_1^2 - n_1 x_2^2). 
\end{align*}
It follows from Lemma \ref{le:12vs22} that
\begin{align*}
    \mathcal{U}_1(1,2), \mathcal{U}_1(2,1) =  \mathcal{U}_{1}(2,2) + O(NX^{2/3} (\log X)^{-C}).
\end{align*}
Therefore, the current task is reduced to showing that
\begin{align} \label{eq:1122claim}
  |\mathcal{U}_{1}(1,1) - \mathcal{U}_{1}(2,2)|  \ll_C  NX^{2/3} (\log X)^{-C}.
\end{align}
We partition according to the sign of $n_2 x_1^2-n_1 x_2^2$ and separate the part where it is small. For $\delta_1 = (\log X)^{-C_1}$ with $C_1 > 0$ large we let $F_0(r)$ be a 1-bounded smooth even function with $F_0(r)=1$ for $|r| \leq \delta_1 NX$ and supported on $|r| \leq 2\delta_1 N.$ We define non-negative functions $F_{\epsilon}(r)$ by $F_{\epsilon}(r)^2 = (1-F_0(r)^2) \mathbf{1}_{\operatorname{sgn}(r) = \epsilon}$. We insert the smooth partition of unity
\begin{align*}
    1 =  \sum_{\epsilon \in \{\pm,0\}} F_{\epsilon}(r)^2, \quad \quad
    \mathcal{U}_1(j,j) = \sum_{\epsilon \in \{\pm,0\}}    \mathcal{U}_\epsilon(j),
\end{align*}
where for $\epsilon \in \{\pm,0\}$
\begin{align*}
  \mathcal{U}_\epsilon(j) &=   \sum_{\substack{n_1,n_2 \sim N \\\gcd(n_1,n_2)=1}} \beta_{ n_1} \beta_{ n_2} \sum_{m} F_{\epsilon}(m) \Upsilon^{j,j}_{n_1,n_2}(m) \mathcal{Q}_{\epsilon,n_1,n_2}(m), \\
\mathcal{Q}_{\epsilon,n_1,n_2}(m) &= \sum_{\substack{x_1,x_2 \\    a (n_2 x_1^2-n_1 x_2^2) = m}}  F_{\epsilon} (a(n_2 x_1^2-n_1 x_2^2)) 
f_1(\tfrac{x_1}{A})f_1(\tfrac{x_2}{A}). 
\end{align*}
We will bound the contribution from $\epsilon=0$ in  Section \ref{sec:smallcontribution}. For a fixed  $\sgn(m)=\epsilon$, $\epsilon \in \{\pm\}$ and $a|m$ we have 
\begin{align*}
     \mathcal{Q}_{\epsilon,n_1,n_2}(m) &= \sum_{\substack{x_1,x_2 \\     |n_2 x_1^2-n_1 x_2^2| = |m/a|}}  F_{\epsilon} (a(n_2 x_1^2-n_1 x_2^2)) 
f_1(\tfrac{x_1}{A})f_1(\tfrac{x_2}{A}).
\end{align*}

\subsection{Evaluation of  $\mathcal{U}_{\epsilon}(1,1)$ and $\mathcal{U}_{\epsilon}(2,2)$}

Applying Lemma \ref{le:binaryrep} with $m= |m/a|$, $Z=AN$, and the smooth function $F(x_1,x_2)=F_{\epsilon,n_1,n_2}(x_1,x_2)$ supported on $(x_1,x_2) \in [1/4,4]^2$ defined by
\begin{align*}
       F_{\epsilon,n_1,n_2}(x_1,x_2) := F_{\epsilon} \Bigl(\tfrac{a A^2N^2}{2n_2}(  x_1^2-x_2^2) \Bigr)   f_1(\tfrac{x_1 N}{2n_2})  f_1(\tfrac{x_2 N}{\sqrt{4n_1n_2}}),
\end{align*}
we obtain for $\epsilon \in \{\pm\}$
\begin{align*}
    \mathcal{U}_\epsilon(j) =  \mathcal{M}_\epsilon(j) + \mathcal{E}_\epsilon(j) + O_\eta(X^{-100}),
\end{align*}
where for $T=X^{1/3-\eta/2}$
\begin{align*}
 \mathcal{M}_\epsilon(j) &=  \frac{1}{2\pi } \int_{|t| \leq X^{\eta^3}}\sum_{\substack{n_1,n_2 \sim N \\\gcd(n_1,n_2)=1}}\frac{\beta_{ n_1} \beta_{ n_2} (\tfrac{a}{2n_2})^{it/2}(AN)^{it}\check{F}_{\epsilon,n_1,n_2}(0,it)}{h(4 n_1 n_2) R_{4n_1n_2}} \mathcal{C}_{\epsilon,n_1,n_2,t}(j) \d t  ,\\
\mathcal{C}_{\epsilon,n_1,n_2,t}(j)  &:=  \sum_{\substack{m}}  F_{\epsilon}(m)\Upsilon_{n_1,n_2}^{j,j}(m)\lambda_{1}^\sharp ( |m/a|,T) |m|^{-it/2} 
\end{align*}
and
\begin{align*}
     \mathcal{E}_\epsilon(j) = & \frac{1}{2\pi } \int_{|t| \leq X^{\eta^3}}\sum_{\substack{n_1,n_2 \sim N \\\gcd(n_1,n_2)=1}}\frac{\beta_{ n_1} \beta_{ n_2} a^{it/2}(AN)^{it}}{h(4 n_1 n_2) R_{4n_1n_2}}  
 \\
 & \hspace{50pt}\times \sum_{|\ell| \leq X^{\eta^3} R_{4n_1n_2}} \sum_{\chi \in \widehat{\CC(d)}}  \check{F}_{\epsilon,n_1,n_2}(\ell,it)\psi_{\chi \xi^{\ell} ,s}(\mathfrak{n}_2) \mathcal{C}^\flat_{\epsilon,n_1,n_2,t}(j,\chi \xi^{\ell} )\d t  ,\\
\mathcal{C}^\flat_{\epsilon,n_1,n_2,t}(j,\chi \xi^{\ell} ) := & \sum_{\substack{m}}  F_{\epsilon}(m)\Upsilon_{n_1,n_2}^{j,j}(m) \lambda_{\chi \xi^{\ell} }^\flat ( |m/a|) |m|^{-it/2}. 
\end{align*}
Note that by \eqref{eq:Fdecay} we have the trivial bound 
\begin{align*}
    \int_{|t| \leq X^{\eta^3}}|\check{F}_{\epsilon,n_1,n_2}(\ell,it)| \ll (\log X)^{O(1)}.
\end{align*}
By the class number formula \eqref{eq:classnumberformula}, Cauchy-Schwarz, and Lemma \ref{le:Linverse}, we have for $j \in \{1,2\}$
\begin{align*}
     \mathcal{M}_\epsilon(j)  &=  \mathcal{M}_\epsilon(j,K)  + O( N X^{2/3-\eta^3}),  \quad K:=X^{\eta^2},\\
     \mathcal{M}_\epsilon(j,K) &=  \frac{1}{2\pi } \int_{|t| \leq X^{\eta^3}}\sum_{\substack{n_1,n_2 \sim N \\\gcd(n_1,n_2)=1}}\frac{ \beta_{ n_1} \beta_{ n_2} a^{it/2}(AN)^{it}\check{F}_{\epsilon,n_1,n_2}(0,it)}{\sqrt{n_1 n_2}} \sum_{k \leq K} \frac{\mu(k) (\tfrac{4n_1 n_2}{k})}{k} \mathcal{C}_{\epsilon,n_1,n_2,t}(j) \d t.
\end{align*}

Recall that the aim is to show \eqref{eq:1122claim}. To this end, in the following subsections we will bound the error terms $E_\epsilon(j) \ll N X^{2/3-\eta^2}$ and show that the main terms match up to a negligible error term $|\mathcal{M}_\epsilon(1,K)  -\mathcal{M}_\epsilon(2,K)| \ll_C  N X^{2/3} (\log X)^{-C}$.  

\subsection{Bounding $\mathcal{E}_\epsilon(1)$}
We invoke Conjecture $\mathrm{C}_a(\eps)$ with the binary cubic form $C(X,Y) = b n_1 X^3- bn_2 Y^3$ to get for $\chi \xi^{\ell} \neq 1$ by summation by parts
\begin{align*}
\mathcal{C}^\flat_{\epsilon,n_1,n_2,t}(1,\chi \xi^{\ell} ) \ll X^{O(\eta^2)}B^{1+2\eps} \ll X^{1/3+2\eps/3+O(\eta^2)}.
\end{align*}
and for $\chi \xi^{\ell} = 1$
\begin{align*}
\mathcal{C}^\flat_{\epsilon,n_1,n_2,t}(1,1) \ll X^{2/3-\eta/4+O(\eta^2)}.
\end{align*}
Therefore, by the condition $N \leq X^{1/3-2\eps/3-\eta}$ we get 
\begin{align*}
    \mathcal{E}_\epsilon(1)  \ll N^2 X^{1/3+2\eps/3+O(\eta^2)}  + N X^{2/3-\eta/4+O(\eta^2)}\ll N X^{2/3-\eta/4},
\end{align*}
\begin{remark} \label{remark:largesieve}
    Conjecture $\mathrm{C}_a(\eps)$  may be replaced by the following large sieve type bound, which saves a factor of $N^{1+2\eta}$ over the trivial bound. \\
    
   \noindent \textbf{Conjecture $\mathrm{L}_{a,b}(\eps)$.} For any $\sqrt{B_1 + B_2} < N \leq (B_1 + B_2)^{1-2\eps}$ we have for some $\eta>0$
    \begin{align*}
 \sum_{\substack{n_1,n_2 \leq N  \\
n_1n_2 \, \text{square-free odd}}} \sum_{\chi \in \widehat{\CC(4n_1n_2)}} \sum_{|\ell| \leq N^{\eta} R_{4n_1n_2}}  &\bigg| \sum_{\substack{y_1 \leq B_1, \, y_2 \leq B_2 \\ n_1y_2^3  \equiv n_2 y_1^3 \pmod{a}}} \lambda^\flat_{\chi \xi^{\ell} }(\tfrac{b}{a}|n_1y_2^3 -n_2 y_1^3|)\bigg| \ll  N^{2-\eta} (B_1+B_2)^2.
\end{align*} 
\end{remark}
\subsection{Bounding $\mathcal{E}_\epsilon(2)$}
Similar to the previous section, it suffices to show that for $\chi \xi^{\ell} \neq 1$ \begin{align} \label{eq:claimC12}
 \mathcal{C}^\flat_{\epsilon,n_1,n_2,t}(2,\chi \xi^{\ell} )  \ll X^{1/3+2\eps/3+O(\eta^2)}.
\end{align}
and for $\chi \xi^{\ell} = 1$
\begin{align} \label{eq:claimC12flat}
\mathcal{C}^\flat_{\epsilon,n_1,n_2,t}(2,1)\ll X^{2/3-\eta/4+O(\eta^2)}.
\end{align}
The bound \eqref{eq:claimC12} follows from combining Lemma \ref{le:22smooth} and Lemma \ref{le:smooothlambda}. The bound \eqref{eq:claimC12flat} follows from  combining Lemma \ref{le:22smooth} and Lemma \ref{le:lambdaflatsmooth} using $T=X^{1/3-\eta/2}$.
\subsection{Evaluation of $\mathcal{M}_\epsilon(2,K)$} \label{sec:m2eval}
Denote $F_{\epsilon,t}(m)= F_{\epsilon}(m)|m|^{-it/2}$. We have 
\begin{align*}
 \mathcal{C}_{\epsilon,n_1,n_2,t}(2)   =\sum_{c \leq T} (\tfrac{4n_1n_2}{c})&  \sum_{n_1 v_2 \equiv n_2v_1 \pmod{ac}}F_{\epsilon,t}(n_1 v_2 - n_2v_1 ) \gamma^{(2)}_{v_1} \gamma^{(2)}_{v_2}\\
 =   \sum_{c \leq T} (\tfrac{4n_1n_2}{c}) & \sum_{y_1,y_2} f_2(\tfrac{U^{1/3} y_1}{B})f_2(\tfrac{U^{1/3} y_2}{B})  \\
 & \times\sum_{\substack{u_1,u_2 \\  bn_1 y_2^3  u_2 \equiv  bn_2 y_1^3  u_1 \pmod{ac }}} F_2(\tfrac{ u_1}{U})  F_2(\tfrac{u_2}{U}) F_{\epsilon,t}(n_1 u_2y_2^3 - n_2 u_1y_1^3 )\theta_{u_1}^W \theta_{u_2}^W.
\end{align*}
Since the variables $u_j$ are localized to $U(1+O(X^{-\eta^2})$  and $|t| \leq X^{\eta^3}$, we may replace
\begin{align*}
  F_{\epsilon,t}(n_1 u_2y_2^3 - n_2 u_1y_1^3 ) \mapsto  F_{\epsilon,t}( U(n_1 y_2^3 - n_2 y_1^3)) 
\end{align*}
with a negligible error term. We have $\gcd(n_1n_2,ac)=1$ due to $\gcd(n_1n_2,P(W))=1$ and $(\tfrac{4n_1n_2}{c})$. Applying Lemma \ref{le:cheapFLsieve} twice with \eqref{eq:gcdbound2},  we see that $ \mathcal{C}_{\epsilon,n_1,n_2,t}(2)$ is up to negligible error terms equal to
\begin{align*}
  &U^{2/3}  \sum_{c \leq T} (\tfrac{4n_1n_2}{c}) \sum_{y_1,y_2} f_2(\tfrac{U^{1/3} y_1}{B})f_2(\tfrac{U^{1/3} y_2}{B}) F_{\epsilon,t}( U(n_1 y_2^3 - n_2 y_1^3))  \\
  & \hspace{100pt} \times \frac{\#\{u \in (\Z/ac\Z)^\times:  bn_1 y_2^3 u \equiv  bn_2 y_1^3 \pmod{ac }\}}{\varphi(ac)}.  
\end{align*}
The number of solutions for $u$ only depends on $\gcd(by_2^3,by_1^3,ac)$. The contribution from the part where $\gcd(y_1^3y_2^3,ac) \geq X^{\eta^3}$ is negligible by trivial bounds. For $\gcd(y_1^3y_2^3,ac) < X^{\eta^3}$ applying Poisson summation (Lemma \ref{le:poisson}) to the variables $y_1,y_2$ produces the main term
\begin{align*}
 B^2 \mathcal{I}_{\epsilon,t}(n_1,n_2) &\sum_{c \leq T} (\tfrac{4n_1n_2}{c}) \sum_{y_1,y_2 \in \Z/ac \Z}   \frac{\#\{u \in (\Z/ac\Z)^\times:  bn_1 y_2^3 u \equiv  bn_2 y_1^3 \pmod{ac }\}}{(ac)^2\varphi(ac)} , \\
 \mathcal{I}_{\epsilon,t}(n_1,n_2) &:=   \int_{\R^2} f_2(y_1) f_2(y_2)  F_{\epsilon,t}(B^3(n_1 y_2^3 - n_2 y_1^3 ))  \d y_1 \d y_2 \\
 &= (n_1n_2)^{-1/3}\int_{\R^2} f_2(\tfrac{y_1}{n_2^{1/3}}) f_2(\tfrac{y_2}{n_1^{1/3}})  F_{\epsilon,t}(B^3( y_2^3 - y_1^3 ))  \d y_1 \d y_2.
\end{align*}
We have $\gcd(n_1n_2,ac)=1$ for $\gcd(n_1n_2,P(W))=1$ and $(\tfrac{4n_1n_2}{c}) \neq 0$. Let $\mathrm{g}:=\gcd(by_2^3,ac)=\gcd(by_1^3,ac)$. Expanding into Dirichlet characters modulo $\tfrac{ac}{\mathrm{g}}$, the sum over $u$ picks out the contribution from the principal character, so that the last expression is equal to
\begin{align*}
 B^2 \mathcal{I}_{\epsilon,t}(n_1,n_2) &\sum_{c \leq T} (\tfrac{4n_1n_2}{c}) \sum_{y_1,y_2 \in \Z/ac \Z}   \frac{1}{(ac)^2 \varphi(\tfrac{ac}{\mathrm{g}})} ,
 \end{align*}
Therefore, we obtain
\begin{align} 
\label{eq:m2main}
\begin{split}
     \mathcal{M}_\epsilon(2,K) &= B^2 \sum_{\substack{n_1,n_2 \sim N \\\gcd(n_1,n_2)=1}}\frac{\beta_{ n_1} \beta_{ n_2} \mathcal{J}_\epsilon(n_1,n_2)}{\sqrt{n_1n_2}} \sum_{k\leq K}\frac{\mu(k) }{k}  \sum_{c \leq T} (\tfrac{4n_1n_2}{c k}) \\
    &\hspace{50pt} \times \sum_{\substack{y_1,y_2 \in \Z/ac \Z \\ \gcd( by_2^3,ac) =  \gcd(by_1^3,ac) }}   \frac{1}{(ac)^2 \varphi(\tfrac{ac}{\mathrm{g}})}  +O(NX^{2/3}(\log X)^{-C}),\\
    \mathcal{J}_\epsilon(n_1,n_2) &:= \frac{1}{2\pi}\int_{\R} (\tfrac{a}{2n_2})^{it/2}(AN)^{it}\check{F}_{\epsilon,n_1,n_2}(0,it) \mathcal{I}_{\epsilon,t}(n_1,n_2) \d t.
\end{split}
\end{align}
Note that, by differentiation under integration, the weight $   \mathcal{J}_\epsilon(n_1,n_2)$ satisfies for all $J_1,J_2 \geq 0$
\begin{align} \label{eq:Jsigmaderivbound}
    \partial_{w_1}^{J_1}   \partial_{w_2}^{J_2} \mathcal{J}_\epsilon(w_1,w_2) \ll_{J_1,J_2} |w_1|^{-J_1}|w_2|^{-J_2} (\log X)^{O(J_1+J_2)}
\end{align}

\subsection{Evaluation of $\mathcal{M}_\epsilon(1,K)$}
We have 
\begin{align*}
 \mathcal{C}_{\epsilon,n_1,n_2,t}(1)   &=\sum_{c \leq T} (\tfrac{4n_1n_2}{c})  \sum_{n_1 v_2 \equiv n_2v_1 \pmod{ac}}F_{\epsilon,t}(n_1 v_2 - n_2v_1 ) \gamma^{(1)}_{v_1} \gamma^{(1)}_{v_2}\\
& =   \sum_{c \leq T} (\tfrac{4n_1n_2}{c})  \sum_{\substack{y_1,y_2 \\  bn_1 y_2^3 \equiv  bn_2 y_1^3 \pmod{ac }}}  f_2(\tfrac{ y_1}{B})f_2(\tfrac{y_2}{B})  F_{\epsilon,t}(n_1 y_2^3 - n_2 y_1^3 ).
\end{align*}
Applying the Poisson summation formula (Lemma \ref{le:poisson}) to $y_1,y_2$ we get 
\begin{align*}
    \mathcal{C}_{\epsilon,n_1,n_2,t}(1)  =  B^2 \mathcal{I}_{\epsilon,t}(n_1,n_2) \sum_{c \leq T} (\tfrac{4n_1n_2}{c})  \frac{\#\{(y_1,y_2) \in (\Z/ac\Z)^2:bn_1 y_2^3 \equiv  bn_2 y_1^3 \pmod{ac }  \}}{{(ac)^2}}
  + O(X^{-100}).
\end{align*}
We have $\gcd(n_1n_2,ac)=1$ for $\gcd(n_1n_2,P(W))=1$ and $(\tfrac{4n_1n_2}{c}) \neq 0$. Expanding into Dirichlet characters modulo $\tfrac{ac}{\mathrm{g}}$ with $\mathrm{g} = \gcd( by_2^3,ac) =  \gcd(by_1^3,ac)$ we have
\begin{align*}
    &\#\{(y_1,y_2) \in (\Z/ac\Z)^2:bn_1 y_2^3 \equiv  bn_2 y_1^3 \pmod{ac }  \}  \\
    &=\sum_{\substack{y_1,y_2 \in \Z/ac\Z  \\ \gcd( by_2^3,ac) =  \gcd(by_1^3,ac)  }}  \frac{1}{\varphi(\tfrac{ac}{\mathrm{g}})} \sum_{\substack{\chi \pmod{\tfrac{ac}{ \mathrm{g}}} }} \chi(\tfrac{by_2^3}{\mathrm{g}})  \overline{\chi}(\tfrac{by_1^3}{\mathrm{g}})  \chi(n_1) \overline{\chi}(n_2).
\end{align*}
The contribution from the principal character $\chi=\chi_0$ matches exactly the main term in \eqref{eq:m2main}. Therefore, to show \eqref{eq:1122claim}, it remains to show that the error term from $\chi \neq \chi_0$
\begin{align*}
    \mathcal{L}_\epsilon(1) =  \sum_{k\leq K}\frac{\mu(k)}{k}  \sum_{c \leq T}\sum_{\substack{y_1,y_2 \in \Z/ac\Z  \\ \gcd( by_2^3,ac) =  \gcd(by_1^3,ac) }}   & \frac{1}{(ac)^2 \varphi(\tfrac{ac}{\mathrm{g}})}\sum_{\substack{\chi \pmod{\tfrac{ac}{ \mathrm{g}}} \\ \chi \neq \chi_0 }} \chi(\tfrac{by_2^3}{\mathrm{g}})  \overline{\chi}(\tfrac{by_1^3}{\mathrm{g}})  \\
   & \times\sum_{\substack{n_1,n_2 \sim N \\\gcd(n_1,n_2)=1}}\frac{\beta_{ n_1} \beta_{ n_2} \mathcal{J}_\epsilon(n_1,n_2)}{\sqrt{n_1n_2}}  \chi(n_1) \overline{\chi}(n_2) (\tfrac{4n_1n_2}{ck})
\end{align*}
satisfies
\begin{align} \label{eq:Lsigma1claim}
    \mathcal{L}_\epsilon(1) \ll_C \frac{N}{(\log X)^C} .
\end{align}
\begin{remark}
With more work, $T$ could be increased to $B^{2-\eta^2}$. Indeed, applying the Poisson summation formula the exponential sums may be expressed in terms of Ramanujan sums which have almost complete cancellation. We leave the details to the interested reader since in any case our treatment of $\chi \xi^{\ell} = 1$ is conditional. 
\end{remark}
\subsection{Bounding $\mathcal{L}_\epsilon(1)$}
We denote $b_0=\gcd(b,c)$ and 
\begin{align*}
    \mathrm{g} = \gcd( by_2^3, by_1^3,ac) = b_0 \gcd(y_2^3,y_1^3,\frac{ac}{b_0}) =: b_0 \mathrm{g}_1.
\end{align*}
 Making the change of variables $c \mapsto b_0c$ we get
\begin{align*}
    \mathcal{L}_\epsilon(1) =  \sum_{b_0|b} \sum_{k\leq K}\frac{\mu(k)}{k}  \sum_{\substack{c \leq T/b_0 \\ \gcd(c,\frac{b}{b_0})=1}}&\sum_{\substack{y_1,y_2 \in \Z/acb_0\Z  \\ \gcd( y_2^3,ac) =  \gcd(y_1^3,ac)= \g_1 }}   \frac{1}{(acb_0)^2 \varphi(\tfrac{ac}{\mathrm{g}_1})} \sum_{\substack{\chi \pmod{\tfrac{ac}{ \g_1}}   \\ \chi \neq \chi_0 }} \chi(\tfrac{y_2^3}{\mathrm{g}_1})  \overline{\chi}(\tfrac{y_1^3}{\mathrm{g}_1})    \\
    & \times\sum_{\substack{n_1,n_2 \sim N \\\gcd(n_1,n_2)=1}}\frac{\beta_{ n_1} \beta_{ n_2} \mathcal{J}_\epsilon(n_1,n_2)}{\sqrt{n_1n_2}}  \chi(n_1) \overline{\chi}(n_2) (\tfrac{4n_1n_2}{ck}).
\end{align*}
We now show that the sum over $y_j$ vanishes unless $\chi$ is a cubic character. By induction on $k$ with $p^k || \g_1$ we see that either $p^k||ac$ or $3|k$, that is, for some integers
$\mathrm{g}_2,\mathrm{g}_3 $  we have 
\begin{align*}
   \mathrm{g}_1 = \mathrm{g}_2 \mathrm{g}_3^3 \quad \text{with} \quad \gcd(\g_2,\tfrac{ac}{\g_2})=1.
\end{align*}
For any fixed $\g_2,\g_3,$  and fixed $\chi \pmod{\tfrac{ac}{\g_2 \g_3^3}}$, we can take the sum over $y_1,y_2$ inside and we have, by the change of variables $y_j \mapsto y_j \g_2 \g_3$, recalling that $\gcd(\g_2,\frac{ac}{\g_2})=1$,
\begin{align*}
  \sum_{\substack{y_1,y_2 \in \Z/ac b_0 \Z \\ 
 (y_1^3, ac) =  (y_2^3, ac) =  \mathrm{g}_2\g_3^3 }} \chi(\tfrac{y_2^3}{\mathrm{g}_2\g_3^3})  \overline{\chi}(\tfrac{y_1^3}{\g_2\g_3^3}) =     \sum_{\substack{y_1,y_2 \in \Z/\tfrac{ac b_0}{\g_2 \g_3} \Z  }} \chi(y_2^3)  \overline{\chi}(y_1^3).
\end{align*}
The sum over $y_1,y_2$ vanishes unless $\chi^3 = \chi_0$ and we get
\begin{align*}
    \mathcal{L}_\epsilon(1) =    \sum_{k \leq K} \frac{\mu(k)}{k}\sum_{b_0|b} \sum_{\substack{c \leq T/b_0 \\ \gcd(c,\frac{b}{b_0})=1}}&\sum_{\substack{y_1,y_2 \in \Z/ac b_0 \Z  \\ \gcd( y_2^3,ac) =  \gcd(y_1^3,ac) = \g_1 }}   \frac{1}{(ac)^2 \varphi(\tfrac{ac}{\g_1})}  \sum_{\substack{\chi  \pmod{\tfrac{ac}{ \mathrm{g}_1}} \\ \chi^3 = \chi_0   \\ \chi \neq \chi_0 }}  \\
    & \times \sum_{\substack{n_1,n_2 \sim N \\\gcd(n_1,n_2)=1}}\frac{\beta_{ n_1} \beta_{ n_2} \mathcal{J}_\epsilon(n_1,n_2)}{\sqrt{n_1n_2}}  \chi(n_1) \overline{\chi}(n_2) (\tfrac{4n_1n_2}{b_0 ck}).
\end{align*}
We use summation by parts to $n_1,n_2$ with \eqref{eq:Jsigmaderivbound} to relax the smooth weight $\mathcal{J}_\epsilon(n_1,n_2)/\sqrt{n_1n_2}$. We have for any $\g=\g_2 \g_3^3| q$, $\gcd(\g_2,\tfrac{q}{\g_2}) =1$
\begin{align*}
\sum_{\substack{y_1,y_2 \in \Z/q\Z  \\ \gcd( y_2^3,q) =  \gcd(y_1^3,q) = \g }}  \frac{1}{q^2 \varphi(\tfrac{q}{\g})}&\leq  d(q) \sum_{\substack{y \in \Z/q\Z  \\ \gcd( y^3,q) = \g }}  \frac{\g_3^2}{ q^2} \leq  d(q) \frac{\g_3}{q \g_2} 
\end{align*}
We glue together $a,b_0,c,k$ to $q=ab_0 ck$ and make the change of variables $ q\mapsto \g_2 \g_3^3 q$. Thus, to prove \eqref{eq:Lsigma1claim}, it suffices to show that
\begin{align*}
    \mathcal{S}(N_1,N_2) :=  \sum_{\g_2} 
\sum_{\g_3} \frac{d(\g_2 \g_3)^{O(1)}}{\g_2^2 \g_3^2} \sum_{q < ab K T / \g_2 \g_3^3} \frac{d(q)^{O(1)}}{q}  &\sum_{\substack{\chi  \pmod{q} \\ \chi^3 = \chi_0   \\ \chi \neq \chi_0 }}   \bigg|\sum_{\substack{n_1 \leq N_1 \\ n_2 \leq N_2 \\\gcd(n_1,n_2)=1}}\beta_{ n_1} \beta_{ n_2} \chi(n_1) \overline{\chi}(n_2) (\tfrac{4n_1n_2}{q})\bigg|
\end{align*}
satisfies for any $N_1,N_2 \leq 2N$  and any $C>0$
\begin{align} \label{eq:largesieveclaim}
        \mathcal{S}(N_1,N_2)  \ll_C \frac{N^2}{(\log X)^C}.
\end{align}
We now drop the condition $\gcd(n_1,n_2)=1$, which gives an error term $\ll \frac{N^2}{W^{1/2}}$ since $\gcd(n_1 n_2, P(W))=1$. We can remove the divisor function $d(q)^{O(1)}$ by applying Cauchy-Schwarz 
\begin{align*}
    \sum_{q} \frac{d(q)^{O(1)}}{q}\bigg|\sum_{n_1,n_2}\bigg| \leq \bigg(     \sum_{q} \frac{d(q)^{O(1)}}{q}\bigg|\sum_{n_1}\bigg|^2\bigg)^{1/2} \bigg(     \sum_{q} \frac{1}{q}\bigg|\sum_{n_2}\bigg|^2\bigg)^{1/2} \leq N(\log X)^{O(1)}\bigg(     \sum_{q} \frac{1}{q}\bigg|\sum_{n_2}\bigg|^2\bigg)^{1/2}.
\end{align*}
Therefore, it suffices to show that for any $C >0$ and $N' \leq 2N$
\begin{align*}
\mathcal{S}(N') := 
 \sum_{q < ab K T} \frac{1}{q}&\sum_{\substack{\chi  \pmod{q} \\ \chi^6 = \chi_0   \\ \chi \neq \chi_0 }}\bigg|\sum_{\substack{n \leq N' }}\beta_{n}  \chi(n) \bigg|^2 \ll_C \frac{N^2}{(\log X)^C},
\end{align*}
Partitioning according to $\cond(\chi) =r$ and using $\gcd(n,P(W))=1$ to bound the contribution from $\gcd( \frac{q}{r},n) > 1$, we get
\begin{align*}
      \mathcal{S}(N') &\ll_C   (\log X)^2\mathcal{S}^\ast(N')  + \frac{N^2}{(\log X)^C}, \\
      \mathcal{S}^\ast(N') &:=  \max_{\substack{R \leq 2a  b TK}} \frac{1}{R} \sum_{r \sim R} 
\sum_{\substack{\chi  \pmod{r}^\ast \\ \chi^6 = \chi_0   \\ \chi \neq \chi_0 }}\bigg|\sum_{\substack{n \leq N' }}\beta_{n}  \chi(n)\bigg|^2.
\end{align*}
We let $C_2 > 0$ be large compared to $C$ and split into three cases depending on the size of $R$.

For $R\leq (\log X)^{C_2}$ we have  $ \mathcal{S}^\ast(N') \ll_C N^2 (\log X)^{-C}$ by the Siegel-Walfisz property \eqref{eq:betacondition} since $\chi \neq \chi_0$.

For $(\log X)^{C_2}< R \leq N^{1/3}$ we get by the duality principle and the Poisson summation formula (similar to the proof of Lemma \ref{le:Linverse}), using $N> X^{\eta} R^2$, that
\begin{align*}
    \mathcal{S}^\ast(N') \ll \frac{N^2}{R^2}  \sum_{r_1,r_2 \sim R} \sum_{\substack{\chi_j  \pmod{r_j}^\ast \\ \chi_j^6 = \chi_0   \\ \chi_j \neq \chi_0  \\  \chi_1=\chi_2}} 1 + O(X^{-100}) \ll  \frac{N^2}{R^2} \sum_{r \sim R} d(r)^{O(1)} \ll_C \frac{N^2}{(\log N)^C}
\end{align*}
 once $C_2$ is sufficiently large in terms of $C$.

For $N^{1/3} < R \leq 2a b T K $ we apply the large sieve for sextic characters due to Baier and Young \cite[Theorem 1.5]{baieryoungcubic} to show that
\begin{align*}
   \mathcal{S}_2^\ast(N',R) \ll \frac{X^{o(1)}}{R} (R^{4/3} + R^{1/2} N) N \ll N^{2-\eta/2}
\end{align*}
since $R \ll TK \ll N^{3-3\eta}$. This completes the proof of \eqref{eq:Lsigma1claim} and thus the proof of \eqref{eq:1122claim}. To complete the proof of Proposition \ref{prop:typeii1} it remains to bound the contributions from $|n_2 x_1^2-n_1 x_2^2| \leq 2 \delta_1$ and $n_0 > 1$.
\subsection{Contribution from $|n_2 x_1^2-n_1 x_2^2| \leq 2 \delta_1 NX$} \label{sec:smallcontribution}
By the triangle inequality, we  have
\begin{align*}
  \mathcal{U}_0(j) & \ll   \sum_{\substack{n_1,n_2 \sim N \\\gcd(n_1,n_2)=1}} |\beta_{ n_1} \beta_{ n_2}| \sum_{m \neq 0} F_{0}(m)^2 \Upsilon^{j,j}_{n_1,n_2}(m) \mathcal{Q}_{0,n_1,n_2}(m), \\
\mathcal{Q}_{0,n_1,n_2}(m) &= \sum_{\substack{x_1,x_2 \\    |a (n_2 x_1^2-n_1 x_2^2)| = |m|}} 
f_1(\tfrac{x_1}{A})f_1(\tfrac{x_2}{A}),
\end{align*}
For $m \in \Z_{\neq 0}$ we can majorize the weight appearing in $\mathcal{Q}_{0,n_1,n_2}(m)$ by
\begin{align*}
f_1(\tfrac{x_1}{A})f_1(\tfrac{x_2}{A}) \leq    f_0(\log| \tfrac{x_12 n_2 + x_2\sqrt{4 n_1 n_2}}{x_1 2 n_2-x_2\sqrt{4 n_1 n_2} }|)  .
\end{align*}
for a smooth bump function $f_0:\R\to \C$ supported in $[-2\log X,2\log X]$ and satisfying $f^{(J)} \ll_J 1$. Then by similar arguments as above, using Lemma \ref{le:binaryrepangle} in place of Lemma \ref{le:binaryrep}, assuming Conjecture $\mathrm{C}_{a}(\eps)$ we have
\begin{align*}
    \mathcal{U}_0(j) \ll \mathcal{M}_0(j) + NX^{2/3-\eta/4}.
\end{align*}
We then drop the condition that $m \neq 0$ and apply Poisson summation to $y_1,y_2$, which captures a factor $\delta_1$ from the support of the weight $F_0(m)$. We then bound trivially the number of solutions 
\begin{align*}
    \#\{(y_1,y_2) \in (\Z/ac\Z)^2:bn_1 y_2^3 \equiv  bn_2 y_1^3 \pmod{ac } \} \leq ac  \gcd(bn_1n_2,ac) 
 d(ac)^{O(1)} 
\end{align*}
instead of expanding into Dirichlet characters. Taking $\delta_1 = (\log X)^{-C_1}$ with $C_1$ large in terms of $C,$ we obtain
\begin{align*}
    \mathcal{M}_0(j) \ll \delta_1 N X^{2/3} (\log X)^{O(1)} \ll_C \frac{N X^{2/3}}{(\log X)^C}.
\end{align*}

\subsection{Contribution from $n_0 > 1$} 
\label{sec:n0contribution}
Since $\beta_n$ are supported on $\gcd(n,P(W))=1$, we have $n_0 \geq W$ and $n_1,n_2 \ll N/W$. We may then use a divisor bound \eqref{eq:roughdivisorbound} to get a contribution
\begin{align*}
 \ll    \sum_{\substack{n_1,n_2 \leq 2 N/W \\ \gcd(n_1,n_2) = 1 \\ n_1n_2 > 1 }}  |\beta_{n_1} \beta_{n_2} |\sum_{\substack{x_1,x_2,v_1,y_2 \\  a (n_2 x_1^2-n_1 x_2^2) = b(n_1 v_2^3-n_2 v_1^3) \neq 0 }}   \sum_{\substack{n_0 | ax_1^2+by_1^3  \\ \gcd(n_0,P(W))=1}}
 f_1(\tfrac{x_1}{A})  f_1(\tfrac{x_2}{A}) \gamma_{v_1}^{(i_1)}\gamma_{v_2}^{(i_2)} \\
 \ll  W^{o(1)}  \sum_{\substack{n_1,n_2 \leq 2N/W \\ \gcd(n_1,n_2) = 1 \\ n_1n_2 > 1 }}  |\beta_{n_1} \beta_{n_2} |\sum_{\substack{x_1,y_1,x_2,y_2 \\  |a (n_2 x_1^2-n_1 x_2^2)| = |b(n_1 y_2^3-n_2 y_1^3)| \neq 0 }}   
 f_1(\tfrac{x_1}{A})  f_1(\tfrac{x_2}{A}) \gamma_{v_1}^{(i_1)}\gamma_{v_2}^{(i_2)}.
\end{align*}
By similar arguments as above (using once more Conjecture $\mathrm{C}_{a}(\eps)$), we see that this is bounded by
$ \ll W^{-1/2}  N X^{2/3}.$ This completes the proof of Proposition \ref{prop:typeii1} and thus the proof of Proposition \ref{prop:typeII}. \qed

Proving the upper bounds  in this and the previous section are non-trivial tasks and we needed to employ the Conjecture $\mathrm{C}_a(\eps)$. It is not clear if a correct-order upper bound can be proved unconditionally by more elementary considerations, and this on its own would be an interesting problem.

\section{Proofs of Theorems \ref{thm:technicalasymp} and \ref{thm:lower}} \label{sec:sieve}
\subsection{Proof of Theorem \ref{thm:technicalasymp}}
We let let $\mathcal{W}=(w_n \log n)$ and define $  \mathcal{W}^{(2)} := (w_{n}^{(2)})$ (of sieve dimension 2) by
\begin{align*}
 w_{n}^{(2)}  := w_n \log n \sum_{4X^{1/3} \leq m < 2X^{1/2}} \mu^2(m) \mathbf{1}_{m|n}.
\end{align*}
For any sequence $\mathcal{C}=(c_n)$ we denote
\begin{align*}
    S(\mathcal{C}_d,z) = \sum_{\gcd(n,P(z))=1} c_{dn}, \quad   S(\mathcal{C},z) =   S(\mathcal{C}_1,z).
\end{align*}
The following lemmas follow quickly from the fundamental lemma of the sieve \cite[Corollary 6.10]{odc}, using respectively Propositions \ref{prop:typeI} and \ref{prop:typeI2} to bound the remainder.
\begin{lemma} \label{le:funi1lemma}
Let $W=X^{1/(\log \log X)^2}$ and $d \leq X^{5/9-\eta}$ be square-free. Then
  \begin{align*}
      S(\mathcal{W}_d, W) \ll_{\eta,C} X^{5/6}(\log X)^{-C}.
  \end{align*}
\end{lemma}
\begin{lemma} \label{le:funi2lemma}
Let $W=X^{1/(\log \log X)^2}$ and let $\alpha_d$ be bounded and supported on square-free integers. Then
  \begin{align*}
    \sum_{d \leq X^{1/4-\eta}} \alpha_d  S(\mathcal{W}^{(2)}_d, W)\ll_{\eta,C} X^{5/6}(\log X)^{-C}.
  \end{align*}
\end{lemma}
The goal is to show that
\begin{align*}
  \sum_{p} w_p  \log p =  S(\mathcal{W},2X^{1/2}) \ll \eps AB \widehat{f_1}(0) \widehat{f_2}(0).
\end{align*}
Applying Buchstab's identity twice with $Z=X^{1/6-2\eps/3-\eta}$ we get
\begin{align*}
      S(\mathcal{W},2X^{1/2}) =&   S(\mathcal{W},Z) - \sum_{Z \leq p < 2X^{1/2}}S(\mathcal{W}_p,p) \\
     = &S(\mathcal{W},Z)  - \sum_{Z \leq p < 4X^{1/3}}S(\mathcal{W}_p,Z) - \sum_{4X^{1/3} \leq p < 2X^{1/2}}S(\mathcal{W}_p,p) \\
     & \hspace{100pt}+  \sum_{Z \leq p_2 \leq p_{1} < 4X^{1/3}}S(\mathcal{W}_{p_1p_2},p_2)  \\
     =:& S_1(W) - S_2(W) - S_3(W) + S_4(W).
\end{align*}
By using Propositions \ref{prop:typeI} and \ref{prop:typeII} the first sum satisfies (similar to \cite[Theorem 3.1]{harman} or \cite[Proposition 25]{merikoskipolyprimes}, for instance)
\begin{align*}
    S_1(\mathcal{W}) \ll_C X^{5/6}(\log X)^{-C}.
\end{align*}
To see this, we have for $W=X^{1/(\log \log X)^2}$
\begin{align*}
     S(\mathcal{W},Z) =  \sum_{d|\frac{P(Z)}{P(W)}} \mu(d)  S(\mathcal{W}_d,W).
\end{align*}
For $d< X^{1/6+\eta^2}$ we can use  Lemma \ref{le:funi1lemma}. For $d\geq X^{1/6+\eta^2}$ with $d|P(Z)$  there is some factor $d_1|d$ with $d_1 \in [X^{1/6+\eta^2},X^{1/3-2\eps/3-\eta^2}]$, so that we can use Proposition \ref{prop:typeII} after relaxing cross-conditions. Similarly, the second sum satisfies 
\begin{align*}
    S_2(\mathcal{W})  =  \sum_{Z \leq p < 4X^{1/3}} \sum_{d|\frac{P(Z)}{P(W)}} \mu(d)  S(\mathcal{W}_{d p},W) \ll_C X^{5/6}(\log X)^{-C}
\end{align*}
by using Lemma \ref{le:funi1lemma} or Proposition \ref{prop:typeII} according to whether $d< X^{1/6+\eta^2}$  or $d\geq X^{1/6+\eta^2}$ .

For the fourth sum $S_4(\mathcal{W})$ we get a contribution $\ll_C X^{5/6} (\log X)^{-C}$ by Proposition \ref{prop:typeII} except for the part where  $ p_1 < X^{1/6+\eta^2}$ or $p_1 > X^{1/3-2\eps/3-\eta^2}$. The bad ranges  contribute $\ll \eps AB \widehat{f_1}(0) \widehat{f_2}(0)$ by a simple upper bound sieve (eg. \cite[Theorem 7.1]{odc}), using Proposition \ref{prop:typeI} to bound the remainder. 

It then remains to handle the third term which counts products of two primes, that is,
\begin{align*}
 -S_3(\mathcal{W})  =   -\sum_{4X^{1/3} \leq p < 2X^{1/2}}S(\mathcal{W}_p,p) =  -S(\mathcal{W}^{(2)},X^{1/3})
\end{align*}
By Buchstab's identity
\begin{align*}
     -S_3(\mathcal{W}) =  -S(\mathcal{W}^{(2)},Z) + \sum_{Z \leq p < X^{1/3}} S(\mathcal{W}^{(2)}_p,p) = -S_5(\mathcal{W}) + S_{6}(\mathcal{W})
\end{align*}
We have $S_5(\mathcal{W}) \ll_C  X^{5/6}(\log X)^{-C}$ by a similar argument as for $S_1(\mathcal{W})$, just using Proposition \ref{prop:typeI2} in place of \ref{prop:typeI}.

Finally, for $S_{6}(\mathcal{W})$ we get  $\ll_C X^{5/6} (\log X)^{-C}$ by Proposition \ref{prop:typeII} except for the part where $p_1 < X^{1/6+\eta^2}$ or $p_1 > X^{1/3-2\eps/3-\eta^2}$. The bad ranges contribute $\ll \eps AB \widehat{f_1}(0) \widehat{f_2}(0)$, by using $|w_{n}^{(2)}| \leq 6 (a_n + b_n) \log n$ for $(n,P(Z))$=1 and applying a simple upper bound sieve (eg. \cite[Theorem 7.1]{odc}), using Proposition \ref{prop:typeI} to bound the remainder. This completes the proof of Theorem \ref{thm:technicalasymp}. \qed

\subsection{Proof of Theorem \ref{thm:lower}}.
The bad ranges come from the terms $S_4(\mathcal{W}),S_6(\mathcal{W})$ which have a positive sign. Thus, taking $\eta =o(1)$ in the above argument, we get the  Harman's sieve lower bound
    \begin{align*}
    \sum_{n} \Lambda(n) a_n \geq (1-\mathfrak{D}_4(\eps) - \mathfrak{D}_6(\eps)+o(1)) \sum_{n} \Lambda(n) b_n
    \end{align*}
    where, denoting $I(\eps):=[1/6,1/3-2\eps/3]$ and the Buchstab function by $\omega(u)$ (similar to \cite[Section 7]{merikoskipolyprimes}, for instance),
    \begin{align*}
        \mathfrak{D}_4(\eps) &= \int_{\substack{1/6-2\eps/3 < \alpha_2 < \alpha_1 < 1/3 \\ \alpha_1,\alpha_2, \alpha_1+\alpha_2 \not \in I(\eps)}} \frac{\omega(\frac{1-\alpha_1-\alpha_2}{\alpha_2}) \d \alpha_1 \d \alpha_2}{\alpha_1 \alpha_2^2} \\
          \mathfrak{D}_6(\eps) &= \int_{\substack{1/6-2\eps/3 < \alpha_2  < 1/3 \\ \alpha_2 \not \in I(\eps)}}  \int_{\substack{1/3 < \alpha_1 < 1/2 }} \frac{\omega(\frac{1-\alpha_1-\alpha_2}{\alpha_2}) \omega(\frac{\alpha_1}{\alpha_2}) \d \alpha_1 \d \alpha_2}{\alpha_2^3}.
    \end{align*}
    For $\eps = 1/17$ a numerical computation $\mathfrak{D}_4(\eps) + \mathfrak{D}_6(\eps) < 0.22+0.73= 0.95$, which shows that 
\begin{align*}
 \sum_{n} \Lambda(n) a_n \geq (0.05 + o(1)) \sum_{n} \Lambda(n) b_n. 
\end{align*}
The second part of Theorem \ref{thm:lower} follows immediately from 
Proposition \ref{prop:typeII}, by restricting the $k$-fold convolution to having a factor in $[X^{1/6+\eta},X^{1/6+2\eta}]$ for some $\eta >0$ small. \qed

The value $\eps = 1/17$ in Theorem \ref{thm:lower} can of course be greatly improved by a more careful argument and by further iterations of Buchstab's identity.

\section{Proof of Theorem\ref{thm:mobius}}
By a finer-than-dyadic decomposition it suffices to show that for  $A \in (\delta X^{1/2}, X^{1/2}]$ and $B\in (\delta X^{1/3}, X^{1/3}]$  and for  any $\nu > 0$ we have 
    \begin{align} \label{eq:mobiusclaim}
       S =   \sum_{n} \mu(n) a_n & \ll \nu AB \widehat{f_1}(0)\widehat{f_2}(0).
    \end{align}
We decompose $n$ into the smooth and rough parts
\begin{align*}
 S=\sum_{P^+(n_s) < X^{\nu^2}}   \sum_{P^-(n_r) \geq X^{\nu^2}}  \mu(n_s) \mu(n_r) a_{n_s n_r} = S_{\leq}  + S_0 + S_{>},
\end{align*}
where $S_{\leq}$ has $n_s \leq X^{\nu^3}$ and $S_>$ has $n_s > X^{\nu}$, and $S_0$ restricts to $X^{\nu^3} < n_s \leq X^{\nu}$.

Using a sieve upper bound with Proposition \ref{prop:typeI} to handle the remainder, we have
\begin{align*}
   S_{\leq} = \sum_{n_s \leq X^{\nu^3}}   \sum_{P^-(n_r) \geq X^{\nu^2}}  \mu(n_s) \mu(n_r) a_{n_s n_r} \ll  \frac{ AB \widehat{f_1}(0)\widehat{f_2}(0)}{\nu^2 \log X} \sum_{n_s \leq X^{\nu^3}}  \frac{1}{n_s} \ll \nu  AB \widehat{f_1}(0)\widehat{f_2}(0).
\end{align*}

We then consider $S_{>}$.   By a greedy algorithm $n_s=n_0 n_1$ with $n_0 \in [X^{\nu},X^{2\nu}]$. Therefore, gluing the variables $m=n_1 n_r$, by Proposition \ref{prop:typeI} and  \cite[Lemma 9]{merikoskipolyprimes} we have
\begin{align*}
    S_\nu \leq  \sum_{\substack{P^+(n_0) < X^{\nu^2} \\  n_0 \in  [X^{\nu},X^{2\nu}]}} \mu^2(n_0) \sum_{m} a_{n_0 m} \ll   AB \widehat{f_1}(0)\widehat{f_2}(0) \sum_{\substack{P^+(n_0) < X^{\nu^2} \\  n_0 \in  [X^{\nu},X^{2\nu}]}} \frac{1}{n_0} \ll \nu AB\widehat{f_1}(0)\widehat{f_2}(0).
\end{align*}

We then consider
\begin{align*}
    S_0 =\sum_{\substack{P^+(n_s) < X^{\nu^2} \\ X^{\nu^3} < n_s \leq X^{\nu} }}   \sum_{\gcd(n_r,P(X^{\nu^2})=1}  \mu(n_s) \mu(n_r) a_{n_s n_r}.
\end{align*}
We have the Heath-Brown identity for any $n \leq X$ (see, for instance, \cite[Section 4.1]{mobiusshort})
\begin{align*}
    \mu = \sum_{1 \leq j \leq 3} (-1)^{j-1}\binom{3}{j} (\mu \mathbf{1}_{[0,X^{1/3}]})^{(\ast)j} \ast 1^{(\ast)(j-1)}.
\end{align*}
We apply this to  $\mu(n_r)$ to get
\begin{align*}
    S_0 =  \sum_{1 \leq j \leq 3} (-1)^{j-1}\binom{3}{j} S_j
\end{align*}
The sum $S_1$ is empty since $n_r > X^{1/2}$, say. We claim that
\begin{align} \label{eq:s2claim}
    S_2 \ll  (\eps \nu^{-2}  2^{1/\nu^2} + \nu)  AB\widehat{f_1}(0)\widehat{f_2}(0)  \\ \label{eq:s3claim}
    S_3 \ll (\eps \nu^{-2}  2^{4/\nu^2} + \nu)  AB\widehat{f_1}(0)\widehat{f_2}(0)  
\end{align}
Assuming that these bounds hold, we can complete the proof \eqref{eq:mobiusclaim} and thus of Theorem \ref{thm:mobius}, by choosing $\eps=\nu^3 2^{-4/\nu^2}$.
\subsection{Bounding $S_2$}
We have
\begin{align*}
    S_2 =  \sum_{\substack{P^+(n_s) < X^{\nu^2} \\ X^{\nu^3} < n_s \leq X^{\nu} }}   \sum_{\substack{\gcd(k_1 k_2 m,P(X^{\nu^2})=1  \\ k_1,k_2 \leq X^{1/3}}}  \mu(n_s)\mu(k_1) \mu(k_2) a_{n_s k_1 k_2 m}.
\end{align*}
The part where $k_1 \in [X^{1/3-\eps},X^{1/3}]$ (resp. $k_2 \in  [X^{1/3-\eps},X^{1/3}]$) contribute by glueing together the variables $n=n_s k_2 m$ and by Proposition \ref{prop:typeI}
\begin{align} \label{eq:s2gluebound}
    \ll 2^{1/\nu^2}  \sum_{\substack{\gcd(k_1,P(X^{\nu^2})=1  \\ k_1  \in [X^{1/3-\eps},X^{1/3}]}} \sum_{n} a_{k_1 n} \ll \eps \nu^{-2}  2^{1/\nu^2}  AB\widehat{f_1}(0)\widehat{f_2}(0).
\end{align}
For $k_1 \in [X^{1/6+\eps},X^{1/3-\eps}]$ or $k_2 \in [X^{1/6+\eps},X^{1/3-\eps}]$ we get by Proposition \ref{prop:typeII} a contribution $\ll_C (\log X)^{-C} AB\widehat{f_1}(0)\widehat{f_2}(0).$ For $k_1,k_2 \leq X^{1/6+\eps}$ we get by the fundamental lemma of the sieve \cite[Corollary 6.10]{odc} to $\gcd(m,P(X^{\nu^2})=1$, using  Propositions \ref{prop:typeI} to bound the remainder, and capturing oscillations in $\mu(n_s)$ by the Prime number theorem,
\begin{align*}
   \ll_C    \frac{ AB\widehat{f_1}(0)\widehat{f_2}(0)}{(\log X)^C} + \nu^{-3} e^{-\frac{1}{10 \nu^2}} AB\widehat{f_1}(0)\widehat{f_2}(0) \ll \nu  AB\widehat{f_1}(0)\widehat{f_2}(0).
\end{align*}
\subsection{Bounding $S_3$}
We have
\begin{align*}
    S_3 =  \sum_{\substack{P^+(n_s) < X^{\nu^2} \\ X^{\nu^3} < n_s \leq X^{\nu} }}   \sum_{\substack{\gcd(k_1 k_2 k_3 m_1 m_2,P(X^{\nu^2})=1  \\ k_1,k_2,k_3 \leq X^{1/3}}}  \mu(n_s)\mu(k_1) \mu(k_2) \mu(k_3) a_{n_s k_1 k_2 k_3 m_1 m_2}.
\end{align*}
Similar to \eqref{eq:s2gluebound}, for the parts where $k_j \in [X^{1/3-\eps},X^{1/3}]$ we glue together the rest of the variables and use Proposition \ref{prop:typeI} to get the bound
\begin{align*}
     \ll  2^{4/\eta^2} \sum_{\substack{\gcd(k_1,P(X^{\nu^2})=1  \\ k_1  \in [X^{1/3-\eps},X^{1/3}]}} \sum_{n} a_{k_1 n} \ll \eps \nu^{-2}  2^{4/\nu^2}  AB\widehat{f_1}(0)\widehat{f_2}(0).
\end{align*}
For $k_j \in [X^{1/6+\eps},X^{1/3-\eps}]$ we may use Proposition \ref{sec:typeii}. For $k_{1},k_2,k_3 \leq X^{1/6+\eps}$ we have $k_{i} k_{j} \leq X^{1/3+2\eps}$ for any $i\neq j$. The parts where $k_i k_j \in [X^{1/3-\eps},X^{1/3+2\eps}]$ contribute
\begin{align*}
\ll  2^{3/\eta^2} \sum_{\substack{\gcd(k_1k_2,P(X^{\nu^2})=1  \\ k_1k_2  \in [X^{1/3-\eps},X^{1/3+2\eps}]}} \sum_{n} a_{k_1 n}  \ll  \eps \nu^{-4}  2^{3/\nu^2}  AB\widehat{f_1}(0)\widehat{f_2}(0).
\end{align*}
For $k_i k_j \in [X^{1/6+\eps},X^{1/3-\eps}]$ we may use Proposition \ref{sec:typeii}. 

In the remaining parts we have $k_1k_2k_3 \leq X^{1/3+3\eps}$. The parts where $k_i k_j \in [X^{1/3-\eps},X^{1/3+2\eps}]$ contribute
\begin{align*}
\ll  2^{2/\eta^2} \sum_{\substack{\gcd(k_1k_2k_2,P(X^{\nu^2})=1  \\ k_1k_2 k_3  \in [X^{1/3-\eps},X^{1/3+3\eps}]}} \sum_{n} a_{k_1 n}  \ll  \eps \nu^{-6}  2^{2/\nu^2}  AB\widehat{f_1}(0)\widehat{f_2}(0).
\end{align*}
For $k_1 k_2 k_3 \in [X^{1/6+\eps},X^{1/3-\eps}]$ we may use Proposition \ref{sec:typeii}.

It then remains to handle
\begin{align*}
 S_4=   \sum_{\substack{P^+(n_s) < X^{\nu^2} \\ X^{\nu^3} < n_s \leq X^{\nu} }}   \sum_{\substack{\gcd(k_1 k_2 k_3 m_1 m_2,P(X^{\nu^2})=1  \\ k_1k_2k_3 \leq X^{1/6+\eps}}}  \mu(n_s)\mu(k_1) \mu(k_2) \mu(k_3) a_{n_s k_1 k_2 k_3 m_1 m_2}.
\end{align*}
By the fundamental lemma of the sieve \cite[Corollary 6.10]{odc} with $\kappa=2$ to $\gcd(m_1 m_2,P(X^{\nu^2})=1$, using  Propositions \ref{prop:typeI2} and \ref{prop:typeI}  to bound the remainder, and capturing oscillations in $\mu(n_s)$ by the Prime number theorem, we get
\begin{align*}
 S_4  \ll_C    \frac{ AB\widehat{f_1}(0)\widehat{f_2}(0)}{(\log X)^C} + \nu^{-5} e^{-\frac{1}{100 \nu^2}} AB\widehat{f_1}(0)\widehat{f_2}(0) \ll \nu  AB\widehat{f_1}(0)\widehat{f_2}(0).
\end{align*} \qed

 \bibliography{quadcubic}
\bibliographystyle{abbrv}
\end{document}